\documentclass[a4paper]{amsart}

\usepackage[english]{babel}
\usepackage[utf8]{inputenc}
\usepackage{amsmath}
\usepackage{amsthm}
\usepackage{amssymb}
\usepackage{graphicx}
\usepackage{enumitem}
\usepackage{tikz-cd}
\usepackage{fullpage}
\usepackage{bbm}
\usepackage{soul}
\usepackage{hyperref}
\setuldepth{e}

\usepackage{mathtools}

\makeatletter
\newtheorem*{rep@theorem}{\rep@title}
\newcommand{\newreptheorem}[2]{%
\newenvironment{rep#1}[1]{%
 \def\rep@title{#2 \ref{##1}}%
 \begin{rep@theorem}}%
 {\end{rep@theorem}}}
\makeatother

\usepackage[backend=biber,style=alphabetic, maxnames=100]{biblatex}
\usepackage{csquotes}

\renewcommand{\O}{\mathcal{O}}

\newcommand{\QQ}{\mathbf{Q}}
\newcommand{\QQp}{\mathbf{Q}_p}
\newcommand{\QQpb}{\overline{\mathbf{Q}}_p}
\newcommand{\ZZ}{\mathbf{Z}}

\newcommand{\FFp}{\mathbf{F}_p}
\newcommand{\FFpb}{\overline{\mathbf{F}}_p}

\newcommand{\Gal}{\mathrm{Gal}}
\newcommand{\Hom}{\mathrm{Hom}}
\newcommand{\HT}{\mathrm{HT}}
\newcommand{\GL}{\mathrm{GL}}
\newcommand{\1}{\mathbf{1}}
\newcommand{\chicyc}{\chi_{\mathrm{cyc}}}
\newcommand{\ovr}[1]{\overline{#1}}

\newcommand\isomto{\stackrel{\sim}{\smash{\longrightarrow}\rule{0pt}{0.4ex}}}

\theoremstyle{definition}
\newtheorem{defn}{Definition}[section]

\theoremstyle{plain}
\newtheorem{thm}[defn]{Theorem}
\newtheorem{lem}[defn]{Lemma}
\newtheorem{prop}[defn]{Proposition}
\newreptheorem{prop}{Proposition}
\newtheorem{cor}[defn]{Corollary}

\newtheorem{hypo}[defn]{Hypothesis}

\newtheorem{thmx}{Theorem}

\theoremstyle{remark}
\newtheorem{rem}[defn]{Remark}

\title{Packets of Serre weights for generic locally reducible two-dimensional Galois representations}
\author{Misja F.A. Steinmetz}
\email{m.f.a.steinmetz@math.leidenuniv.nl}

\begin{document}

\begin{abstract}
Suppose $K/\QQp$ is finite and $\ovr{r}\colon G_K\to \GL_2(\FFpb)$ is a reducible Galois representation. In this paper we prove that we can use the results from \cite{ste22} to obtain a decomposition of the set of Serre weights $W(\ovr{r})$ into a disjoint union of at most $(e+1)^f$ `packets' of weights (where $f$ is the residue degree and $e$ the ramification degree of $K$) under the assumption that $\ovr{r}$ is weakly generic. Thereby, we improve on results of  \cite{ds15} which give a similar decomposition, by rather different methods, under the assumption that $\ovr{r}$ is strongly generic. We show that our definition of weak genericity is optimal for the results of this paper to hold when $e=1$. However, we expect that for $e=2$ one of the main results of this paper still holds under weaker hypotheses than the ones in this paper.
\end{abstract}

\maketitle

\section{Introduction}

Let $p$ be a prime. In \cite{ser87} Serre conjectured that an irreducible, odd, mod $p$ Galois representation $\ovr{\rho}\colon G_{\QQ}\to \GL_2(\FFpb)$ must arise as the reduction of the Galois representation attached to a modular form. Moreover, Serre gives a precise recipe for the weight and level of the modular form in terms of $\ovr{\rho}$. Generalisations of Serre's conjecture to the case of Hilbert modular forms over a totally real number field $F$ and representations $\ovr{\rho}\colon G_F\to\GL_2(\FFpb)$ have been the subject of much research in the last decade. Initally, \cite{bdj10} gave a generalisation for such $F$ under the assumption that $p$ is unramified in $F$. This assumption was later removed by \cite{sch08}, \cite{gee11} and \cite{blgg13}. Similarly to the original conjecture, these conjectures attach a set of weights $W(\ovr{\rho})$ to the representation for which $\ovr{\rho}$ is conjectured to be modular. The so-called weight part of the conjecture was resolved in \cite{gls15} building on work of \cite{blgg13}, \cite{gk14} and \cite{new14}. This says that if the representation is modular, then it must be modular for a weight in $W(\ovr{\rho})$. A general aspect of these conjectures is that $\sigma\in W(\ovr{\rho})$ is defined as $\sigma = \otimes_{v\mid p}\sigma_v$ with $\sigma_v\in W(\ovr{\rho}|_{G_{F_v}})$, for all places $v\mid p$ of $F$, where $F_v$ is the completion of $F$ at $v$. Therefore, we may, and will, work completely locally.

Suppose then that $K/\QQp$ is a finite extension of residue degree $f$ and ramification degree $e$. Given $\ovr{r}\colon G_K\to\GL_2(\FFpb)$, we would like to study the set of Serre weights $W(\ovr{r})$ associated to $\ovr{r}$. Since $W(\ovr{r})$ is explicitly defined when $\ovr{r}$ is irreducible (see Defn.~\ref{defn:w-exp-semisimple}), we will assume $\ovr{r}$ is reducible
\[
\ovr{r}\sim \begin{pmatrix} 
\chi_1 & c_{\ovr{r}} \\ 0 & \chi_2
\end{pmatrix},
\]
for characters $\chi_1,\chi_2\colon G_K \to \FFpb^\times$. The class $c_{\ovr{r}}$ is an element of $\mathrm{Ext}^1_{\FFpb[G_K]}(\chi_2,\chi_1) = H^1(G_K,\FFpb(\chi_1\chi_2^{-1})$. 

In fact, if $c_{\ovr{r}} = 0$, then $W(\ovr{r})$ is also explicitly defined (see Defn.~\ref{defn:w-exp-semisimple}). If this is not the case, then $W(\ovr{r})$ is defined as follows. Given a Serre weight $\sigma\in W(\ovr{r}^\mathrm{ss})$ (see Defn.~\ref{defn:w-exp-semisimple}), we usually define a subspace $L_\sigma(\chi_1,\chi_2)\subseteq H^1(G_K,\FFpb(\chi_1\chi_2^{-1})$ in terms of reductions of crystalline representations of specified Hodge--Tate weights (see Defn.~\ref{defn:L_sigma-chi1-chi2}). Then we define $\sigma\in W(\ovr{r})$ if and only if $c_{\ovr{r}} \in L_\sigma(\chi_1,\chi_2)$. Unfortunately, the definition of $L_\sigma(\chi_1,\chi_2)$ in terms of $p$-adic Hodge theory is neither explicit nor computable. The approach of \cite{ddr16}, \cite{cegm17} and \cite{ste22} to make this explicit is to use local class field theory and the Artin--Hasse exponential to define explicit basis elements $c_{i,j}$ of $H^1(G_K,\FFpb(\chi_1\chi_2^{-1})$ for $0\le i <f$ and $0\le j <e$; an alternative approach to making the subspace $L_\sigma(\chi_1,\chi_2)$ explicit using Kummer theory can be found in \cite{bs22}. In \cite{ste22} an indexing set $J_\sigma(\chi_1,\chi_2)$ is defined (see Defn.~\ref{defn:J^AH}) and it is proved that $L_\sigma(\chi_1,\chi_2) = \mathrm{Span}(\{c_{i,j}\mid (i,j)\in J_\sigma(\chi_1,\chi_2)\})$. Unfortunately, the definition of $J_\sigma(\chi_1,\chi_2)$ is combinatorial and not as straightforward as one would have hoped. The first main result of this paper is that under the assumption that $\ovr{r}$ is weakly generic (see Hypo.~\ref{hypo:generic}), the definition of $J_\sigma(\chi_1,\chi_2)$ can be greatly simplified.

\begin{thmx}
Suppose $\ovr{r}$ is weakly generic and $\sigma\in W(\ovr{r}^\mathrm{ss})$. Then there exists an $\ell\in \ZZ^f$ (depending on $\ovr{r}$ and $\sigma$) with $\ell_i\in [0,e]$, for all $i$, such that
\[
L_\sigma(\chi_1,\chi_2) = \mathrm{Span}(\{c_{i,j}\mid j<\ell_i\}).
\]
\end{thmx}

This is Thm.~\ref{thm:gen-conj} below. The proof of this theorem positively resolves Conjecture 7.4.8 from \cite{ste20}. Then we go on to use this result to prove a decomposition of $W(\ovr{r})$. For each $w\in \ZZ^f$ with $w_i\in [0,e]$ for all $i$, we define a subset $P_w \subseteq W(\ovr{r}^\mathrm{ss})$ that is typically of size $2^{f-\delta_w}$ for $\delta_w = |\{i \mid w_i = 0\text{ or }w_i=e\}|$. Order the set of all such $w$ by the natural product ordering, i.e. define $\le$ by
\[
(w_0,w_1,\dotsc,w_{f-1}) \le (w_0',w_1',\dotsc,w_{f-1}')\text{ if and only if } w_i\le w_i'\text{ for all $i$.}
\]

\begin{thmx}
Suppose $\ovr{r}$ is weakly generic. Then we have
\[
W(\ovr{r}) = \coprod_{w\le w^\mathrm{max}} P_w
\]
for some $w^\mathrm{max}$ depending on $\ovr{r}$.
\end{thmx}

In other words, this theorem says that the weights come in packets and these packets respect the natural ordering on the indices $w$. This is Thm.~\ref{thm:mainthm} below.

In \S\ref{sec:rem-on-genericity} we show that our definition of weak genericity is optimal for the results of this paper to hold when $e=1$. We also give a short explanation of why we expect that Thm.~A will still hold under weaker assumptions when $e=2$.

The paper \cite{ds15} inspired us to write this paper and we have certainly borrowed from their clear style and exposition. However, we use completely different methods to obtain similar results and, indeed, the techniques used here were not yet available when \cite{ds15} was written. Moreover, our genericity hypothesis is weaker than that used in \cite{ds15}. To distinguish the two cases we have referred to their hypothesis as strongly generic and to our hypothesis as weakly generic (see Hypo.~\ref{hypo:generic}). In particular, Thm.~\ref{thm:mainthm} and Prop.~\ref{prop:P-w-size-strong-case} recover the main local result of \cite{ds15} when $\ovr{r}$ is strongly generic. In \S\ref{sec:rem-on-genericity} we have included a short discussion of cases in which the results of this paper apply, but that cannot be proved using the methods of \cite{ds15}. This shows that our different approach has resulted in more than could have been achieved by the methods of \cite{ds15}, even if the authors had opted to use more difficult combinatorics. Our weakening of the generic hypothesis does come at the expense of no longer having a easy expression for the size of $P_w$ (cf. Prop.~\ref{prop:P-w-size-strong-case} and Prop.~\ref{prop:P-w-size-weak-case}). A more involved expression for $|P_w|$ may exists also in the weakly generic case -- this essentially boils down to a good understanding of when $(J,x)$ is not maximal for $\sigma$ -- but we have chosen not to pursue this direction for brevity.

We believe that it should be possible to use the techniques \cite{bs22} to recover the results of this paper, but we have not attempted to do this.

\subsection{Acknowledgements} We would like to thank Fred Diamond for his small remark that sparked our investigation into the main results of this paper. We would also like to thank the Mathematical Institute at Leiden University for their support while working on this paper. We would like to thank Robin Bartlett for his comments on an early version of this paper. Naturally, the author is responsible for any remaining errors.

\subsection{Notation}\label{subsec:notation}

We will write $K/\QQp$ for a finite extension with residue degree $f$ and ramification degree $e$. Denote its residue field by $k$. We fix a choice of uniformiser $\pi\in K$. We also fix an algebraic closure $\ovr{K}$ and a $(p^f-1)$-th root $\pi^{1/(p^f-1)}\in \ovr{K}$ of $\pi$. Let $\omega\colon G_K \to k^\times$ denote the character defined by letting $\omega(g)$ equal the image of $\frac{g(\pi^{1/(p^f-1)})}{\pi^{1/(p^f-1)}}$ in $k^\times$. The restriction of $\omega$ to the inertia subgroup $I_K$ of $G_K$ is independent of the choice of $\pi^{1/(p^f-1)}$. For any $\tau\in \Hom_{\FFp}(k,\FFpb)$, we obtain a character $\omega_\tau\colon G_K\to \FFpb^\times$ as $\omega_\tau:=\tau \circ \omega$. Unless otherwise stated we write $\varphi\colon k\to k$ for the $p$-th power map $\varphi(x) = x^p$ (and similarly for other finite fields).

A Serre weight $\sigma$ is an isomorphism class of irreducible $\FFpb$-representations of $\GL_2(k)$. Recall that these are represented by
\[
\sigma_{a,b} = \bigotimes_{\tau \in \Hom_{\FFp}(k,\FFpb)} \left(\mathrm{det}^{b_\tau} \otimes_k \mathrm{Sym}^{a_\tau - b_\tau} k^2\right) \otimes_{k,\tau} \FFpb
\]
for uniquely determined integers $a_\tau,b_\tau$ satisfying $b_\tau,a_\tau - b_\tau \in [0,p-1]$ and not all $b_\tau$ equal to $p-1$. When convenient we may relax this restriction to $b_\tau\in\ZZ$, while still requiring $a_\tau - b_\tau \in [0,p-1]$, for all $\tau$ and use the identification $\sigma_{a,b}\cong \sigma_{a',b'}$ if and only if $a_\tau-b_\tau =a_\tau' - b_\tau '$ for all $\tau$ and $\sum_{i=0}^{f-1}b_{\tau\circ\varphi^i}p^i \equiv \sum_{i=0}^{f-1}b_{\tau\circ\varphi^i}' p^i\bmod{(p^f-1)}$ for any $\tau$.

We write $\chicyc$ for the $p$-adic cyclotomic character. Suppose $V$ is a Hodge--Tate representation of $G_K$ on a $\QQpb$-vector space. For each $\kappa\in \Hom_{\QQp}(K,\QQpb)$, the $\kappa$-Hodge Tate weights $\HT_\kappa(V)$ is the multiset of integers containing $i$ with multiplicity
\[
\dim_{\QQpb}\left( V\otimes_{\kappa,K} \widehat{\ovr{K}}(-i)\right)^{G_K},
\]
where $\widehat{\ovr{K}}(i)$ denotes the completed algebraic closure of $K$ with the twisted $G_K$-action $g(a) = \chicyc(g)^i g(a)$. Thus, $\HT_\kappa(\chicyc) = \{1\}$ for every $\kappa$.

For any tuple $a = (a_\tau)_\tau \in \ZZ^{\Hom_{\FFp}(k,\FFpb)}$, we define
\[
\Omega_{\tau,a} := \sum_{i=0}^{f-1} p^i a_{\tau \circ \varphi^i}.
\]

\section{Explicit Serre weights in the semisimple case}

Let us begin by defining the set of Serre weights in terms of crystalline lifts.

\begin{defn} For a Serre weight $\sigma = \sigma_{a,b}$, we will say a tuple of integers $\tilde{\sigma} = (\tilde{a}_\kappa,\tilde{b}_\kappa)_{\kappa \in \Hom_{\QQp}(K,\QQpb)}$ is a lift of $\sigma$ if, for each $\tau\colon k\to \FFpb$, there is an indexing
\[
\left\{\kappa \in \Hom_{\QQp}(K,\QQpb)\mid \kappa|_{k} = \tau\right\} = \{\tau_0,\tau_1,\dotsc,\tau_{e-1}\}
\]
such that
\[
(\tilde{a}_\kappa,\tilde{b}_\kappa) = \begin{cases}(a_\tau +1, b_\tau) &\text{ if }\kappa = \tau_0; \\(1,0) &\text{ if }\kappa = \tau_i\text{ for }i>0. \end{cases}
\]
We say a crystalline representation of $G_K$ on a finite free $\ovr{\mathbf{Z}}_p$-module $V$ has Hodge type $\sigma$ if there exists a lift $\tilde{\sigma} = (\tilde{a}_\kappa,\tilde{b}_\kappa)$ of $\sigma$ such that
\[
\HT_\kappa(V) = (\tilde{a}_\kappa,\tilde{b}_\kappa)
\] 
for every $\kappa\colon K\to \QQpb$.
\end{defn}

\begin{defn}
Suppose $\ovr{r}\colon G_K\to \GL_2(\FFpb)$ is a continuous representation. We write $W^\mathrm{cr}(\ovr{r})$ for the set of Serre weights $\sigma_{a,b}$ for which there exists a crystalline representation of $G_K$ on a finite free $\ovr{\mathbf{Z}}_p$-module $V$ with Hodge type $\sigma_{a,b}$ such that $V\otimes_{\ovr{\mathbf{Z}}_p} \FFpb \cong \ovr{r}$.
\end{defn}

\subsection{The semisimple case}

Unfortunately, $W^\mathrm{cr}(\ovr{r})$ is neither explicit nor computable. When $\ovr{r}$ is semisimple, there is a simple explicit equivalent definition. Let us recall this definition here. 

For any character $\chi\colon G_K\to \FFpb^\times$, there exist integers $(n_\tau)_\tau \in \ZZ^{\Hom_{\FFp}(k,\FFpb)}$ such that
\[
\chi|_{I_K} = \prod_\tau \omega_\tau^{n_\tau}
\]
and these are unique if we require $n_\tau\in [1,p]$ for all $\tau$ and $n_\tau<p$ for at least one $\tau$. Since $\omega_\tau^p = \omega_{\tau \circ \varphi}$, we can write $\chi|_{I_K} = \omega_\tau^{\Omega_{\tau, n}}$ with $\Omega_{\tau,n}$ defined as in \S\ref{subsec:notation}.

Let $K_2$ denote the unique unramified extension of $K$ of degree 2 with residue field $k_2$ and let $\pi\colon\Hom_{\FFp}(k_2,\FFpb) \to \Hom_{\FFp}(k,\FFpb)$ denote the natural projection induced by restriction to $k$. For $\lambda\in \Hom_{\FFp}(k_2,\FFpb)$, let $\omega_\lambda$ denote the fundamental character of $I_{K_2} = I_{K}$ corresponding to $\lambda$. Following \cite{gls15}, we give the definition $W^\mathrm{exp}(\ovr{r})$ for semisimple $\ovr{r}$.

\begin{defn}\label{defn:w-exp-semisimple}
Suppose $\ovr{r}\colon G_K\to\GL_2(\FFpb)$ is continuous and semisimple. We define the set of Serre weights $W^\mathrm{exp}(\ovr{r})$ as follows.
\begin{enumerate}
\item If $\ovr{r}$ is irreducible, then $\sigma_{a,b}\in W^\mathrm{exp}(\ovr{r})$ if there is a subset $J\subset \Hom_{\FFp}(k_2,\FFpb)$ and integers $x_\tau \in [0,e-1]$, for all $\tau\in \Hom_{\FFp}(k,\FFpb)$, such that
\[
\ovr{r}|_{I_K}\cong \begin{pmatrix}
\prod_{\lambda\in J} \omega_\lambda^{a_{\pi(\lambda)} +1 + x_{\pi(\lambda)}}\prod_{\lambda\notin J}\omega_{\lambda}^{b_{\pi(\lambda)}+e-1-x_{\pi(\lambda)}} & 0 \\
0 & \prod_{\lambda\notin J}\omega_{\lambda}^{a_{\pi(\lambda)} + 1 +x_{\pi(\lambda)}}\prod_{\lambda\in J}\omega_{\lambda}^{b_{\pi(\lambda)}+e-1-x_{\pi(\lambda)}} \end{pmatrix}
\]
and the restriction $\pi\colon J\isomto \Hom_{\FFp}(k,\FFpb)$ is a bijection.

\item If $\ovr{r}$ is the direct sum of two characters, then $\sigma_{a,b}\in W^\mathrm{exp}(\ovr{r})$ if there exists a subset $J\subset \Hom_{\FFp}(k,\FFpb)$ and integers $x_\tau\in [0,e-1]$, for all $\tau\in \Hom_{\FFp}(k,\FFpb)$, such that
\[
\ovr{r}|_{I_K} \cong \begin{pmatrix}
\prod_{\tau\in J} \omega_\tau^{a_\tau + 1 + x_\tau}\prod_{\tau\notin J}\omega_\tau^{b_\tau + x_\tau} & 0 \\
0 & \prod_{\tau\notin J} \omega_\tau^{a_\tau+e-x_\tau}\prod_{\tau \in J} \omega_\tau^{b_\tau + e -1 -x_\tau}
\end{pmatrix}.
\]
\end{enumerate}
\end{defn}

\begin{thm}[Gee--Liu--Savitt, Wang]\label{thm:W^cr=W^exp-semisimple} If $\ovr{r}$ is semisimple, then $W^\mathrm{exp}(\ovr{r}) = W^\mathrm{cr}(\ovr{r})$.
\end{thm}
\begin{proof}
For $p>2$ this is \cite[Thm.~4.1.6]{gls15} and for $p=2$ this is \cite[Thm.~5.4]{wan17}.
\end{proof}

\section{Explicit Serre weights in the non-semisimple case}

Having treated the semisimple case in the previous section, we will now assume $\ovr{r}$ is reducible. Then $\ovr{r}$ is of the form
\[
\ovr{r} \sim \begin{pmatrix}\chi_1 & * \\ 0 & \chi_2 \end{pmatrix}
\]
for characters $\chi_1,\chi_2\colon G_K\to \FFpb^\times$. The results of \cite{gls15} show that we have an inclusion $W^\mathrm{cr}(\ovr{r})\subset W^\mathrm{cr}(\ovr{r}^\mathrm{ss})$. However, this is rarely an equality. To make $W^\mathrm{cr}(\ovr{r})$ explicit we need to specify a condition on the extension class of $\ovr{r}$ which determines whether $\sigma\in W^\mathrm{cr}(\ovr{r}^\mathrm{ss})$ lies in $W^\mathrm{cr}(\ovr{r})$ or not. Note that, for any representation $\ovr{\rho}\colon G_K \to \GL_2(\FFpb)$ of the form
\[
\ovr{\rho} \sim \begin{pmatrix} \chi_1 & c_{\ovr{\rho}} \\ 0 & \chi_2 \end{pmatrix},
\]
the extension class $c_{\ovr{\rho}}$ is an element of $H^1(G_K,\FFpb(\chi))$, where $\chi = \chi_1\chi_2^{-1}$.

\begin{defn}\label{defn:L_sigma-chi1-chi2} Suppose $\sigma = \sigma_{a,b}$ is a Serre weight. Write $L_\sigma(\chi_1,\chi_2)$ for all the extension classes in $H^1(G_K,\FFpb(\chi))$ arising as the reduction of a crystalline representation $r$ of $G_K$ on a finite free $\ovr{\mathbf{Z}}_p$-module $V$ of Hodge type $\sigma$ which are of the form
\[
r\sim \begin{pmatrix} \tilde{\chi}_1 & * \\ 0 & \tilde{\chi}_2 \end{pmatrix},
\]
where $\tilde{\chi}_1$ and $\tilde{\chi}_2$ are crystalline lifts of $\chi_1$ and $\chi_2$, respectively.
\end{defn}

\begin{thm}[Gee--Liu--Savitt, Wang]\label{thm:GLS-W^cr=W^exp} Suppose that $\ovr{r}\sim\begin{pmatrix} \chi_1 & c_{\ovr{r}} \\ 0 & \chi_2 \end{pmatrix}$. We have that $\sigma \in W^\mathrm{cr}(\ovr{r})$ if and only if 
\begin{enumerate}
\item $\sigma\in W^\mathrm{exp}(\ovr{r}^\mathrm{ss})$ and
\item $c_{\ovr{r}} \in L_\sigma(\chi_1,\chi_2)$.
\end{enumerate}
Moreover, whenever $\sigma\in W^\mathrm{exp}(\ovr{r}^\mathrm{ss})$, we have that $L_\sigma(\chi_1,\chi_2)$ is a subspace.
\end{thm}
\begin{proof}
This is the main local result of \cite{gls15} for $p>2$ and follows from \cite{wan17} for $p=2$.
\end{proof}

The approach of \cite{ddr16}, \cite{cegm17} and \cite{ste22} is to make the set of extensions $L_\sigma(\chi_1, \chi_2)$ explicit using local class field theory and the Artin--Hasse exponential. Let us briefly recall this approach.

\subsection{Explicit basis elements} We will start by constructing an explicit basis of $H^1(G_K,\FFpb(\chi))$. Recall that we defined the characters $\omega_\tau$, for $\tau \in \Hom_{\FFp}(k,\FFpb)$, in \S\ref{subsec:notation}. Writing $\chi = \chi_1\chi_2^{-1}$, we have
\[
\chi = \psi \prod_{\tau \in \Hom_{\FFp}(k,\FFpb)} \omega_\tau^{n_\tau}
\]
for an unramified character $\psi$ and integers $n_\tau\in[1,p]$ for all $\tau$. We require that $n_\tau<p$ for at least one $\tau$ so that the $n_\tau$ are uniquely determined. Write $\Omega_{\tau,n}=\sum_{i=0}^{f-1} p^i n_{\tau \circ \varphi^i}$. Since $\omega_{\tau\circ \varphi^i} = \omega_\tau^p$, it follows that $\chi|_{I_K} = \omega_\tau^{\Omega_{\tau,n}}|_{I_K}$ for all $\tau$. Write $f'$ for the smallest integer $i>0$ such that $n_{\tau\circ\varphi^i}= n_\tau$ for all $\tau$. Note that $f'\mid f$ and let $f''= f/f'$. 

For any $j\ge 0$, write $W_j'$ for the set of integers $m\in \ZZ$ satisfying $\frac{jp(p^f-1)}{p-1}<m<\frac{(j+1)p(p^f-1)}{p-1}$, $p\nmid m$ and there exists a $\tau$ such that $m\equiv \Omega_{\tau,n} \bmod{(p^f-1)}$. Then $W_j'$ has cardinality $f'$ (see \cite[Rem.~3.4]{ste22}). Let $W':=\cup_{j=0}^{e-1} W_j'$ and $W:= W' \times \{0,\dotsc,f''-1\}.$ Then $W$ has cardinality $ef$. For any $m\in W'$, we have an embedding $\tau_m \in \Hom_{\FFp}(k,\FFpb)$ such that $m\equiv \Omega_{\tau_m,n} \bmod{(p^f-1)}$. To any $\alpha = (m,k)\in W$, we may attach an embedding $\tau_\alpha:=\tau_m\circ \varphi^{-kf'}$. Notice that $m\equiv \Omega_{\tau_\alpha,n}\bmod{(p^f-1)}$.

Write $\varpi:=\pi^{1/(p^f-1)}$ with $\pi^{1/(p^f-1)}$ as in \S\ref{subsec:notation}. We define $M:=L(\varpi)$, where $L/K$ is a fixed unramified extension of degree prime-to-$p$ such that $\psi|_{G_L}$ is trivial. Note that $\chi|_{G_M}$ is also trivial. Write $l$ for the residue field of $M$ and $\O_M$ for its ring of integers. Moreover, let $\lambda_{\tau, \psi}$ denote a basis element of the 1-dimensional space $(l\otimes_{k,\tau}\FFpb)^{\Gal(L/K) = \psi}$. Lastly, let $E(X):=\exp(\sum_{n\ge 0} X^{p^n}/p^n) \in \ZZ_p[[X]]$. For any $\alpha\in M$ with positive valuation, we define a map
\begin{align*}
\varepsilon_\alpha\colon l \otimes_{\FFp} \FFpb &\to \O_M^\times\otimes_{\ZZ}\FFpb \\
a \otimes b &\mapsto E([a]\alpha) \otimes b,
\end{align*}
where $[a]$ is a Teichm\"uller lift of $a\in l$. For any $\alpha=(m,k)\in W$, we define
\[
u_\alpha := \varepsilon_{\varpi^m}(\lambda_{\tau_\alpha, \psi})\in \O_M^\times\otimes_{\ZZ}\FFpb.
\]
Furthermore, if $\chi$ is trivial (respectively, cyclotomic) we refer to \cite[\S3.2.4]{ste22} for the definitions of $u_\mathrm{triv}\in M^\times \otimes_{\ZZ} \FFpb$ (respectively, $u_\mathrm{cyc}\in \O_M^\times \otimes_{\ZZ} \FFpb$). Since it follows from \cite[\S3.2.2]{ste22} that the isomorphisms of local class field theory give
\[
H^1(G_K, \FFpb(\chi))\cong \Hom_{\FFpb}\left(\left(M^\times\otimes_{\ZZ}\FFpb(\chi^{-1})\right)^{\Gal(M/K)}, \FFpb\right),
\]
the $\FFpb$-dual $c_\alpha$ of $u_\alpha$, for $\alpha\in W$, lies in $H^1(G_K,\FFpb(\chi))$ and similarly for $c_\mathrm{un}$ (respectively, $c_\mathrm{tr}$), the $\FFpb$-dual of $u_\mathrm{triv}$ (respectively, $u_\mathrm{cyc}$), if $\chi$ is trivial (respectively, cyclotomic).

\begin{prop}\label{prop:h-1-basis}
The set $\{c_\alpha\mid \alpha \in W\}$ together with $c_\mathrm{un}$ (respectively, $c_\mathrm{tr}$) if $\chi$ is trivial (respectively, cyclotomic) forms a basis of $H^1(G_K,\FFpb(\chi))$.
\end{prop}

\begin{proof} This is \cite[Cor.~3.7]{ste22}.\end{proof}

\subsection{Explicit sets of extensions} Let us now show how to use our explicit basis elements to get to an explicit definition of $L_\sigma(\chi_1,\chi_2)$. These are results from \cite{ddr16}, \cite{cegm17} and \cite{ste22}, slightly adapted to suit the case at hand.

\begin{defn}\label{defn:S-chi1-chi2-r} Let $\sigma= \sigma_{a,b}$ be a Serre weight. Then we define $\mathcal{S}(\chi_1,\chi_2, \sigma)$ to equal the set of pairs $(J,x)$ with $J\subset \Hom_{\FFp}(k,\FFpb)$ and $x \in \ZZ^{\Hom_{\FFp}(k,\FFpb)}$ with $x_\tau\in [0,e-1]$ for all $\tau$ such that
\[
\ovr{r}^\mathrm{ss}|_{I_K} \cong \begin{pmatrix}
\prod_{\tau\in J} \omega_\tau^{a_\tau + 1 + x_\tau}\prod_{\tau\notin J}\omega_\tau^{b_\tau + x_\tau} & 0 \\
0 & \prod_{\tau\notin J} \omega_\tau^{a_\tau+e-x_\tau}\prod_{\tau \in J} \omega_\tau^{b_\tau + e -1 -x_\tau}
\end{pmatrix}.
\]
\end{defn}

\begin{defn}\label{defn:order-on-S-chi1-chi2-sigma}
Suppose $(J,x)\in \mathcal{S}(\chi_1,\chi_2, \sigma)$. We define $s(J,x)\in\ZZ^{\Hom_{\FFp}(k,\FFpb)}$ as
\[
s(J,x) = \begin{cases} a_\tau - b_\tau + 1 + x_\tau &\text{if $\tau\in J$;} \\
x_\tau &\text{if $\tau\notin J$.}
\end{cases}
\]
We impose an ordering $\preceq$ on $\mathcal{S}(\chi_1,\chi_2, \sigma)$ by stipulating that, for all $(J,x), (J',x')\in \mathcal{S}(\chi_1,\chi_2, r)$,
\[
(J,x) \preceq (J',x') \;\text{if and only if}\; \Omega_{\tau,s(J',x')-s(J,x)}\in (p^f-1)\ZZ_{\ge 0}\text{ for all $\tau$.}
\]
\end{defn}

\begin{prop}\label{prop:max-s-t}
If $\mathcal{S}(\chi_1,\chi_2, \sigma)$ is non-empty, then it contains a unique maximal element $(J^\mathrm{max},x^\mathrm{max})$.
\end{prop}
\begin{proof}
This follows from \cite[Lem.~5.3.3]{gls15} -- see also \cite[Prop.~6.5]{bs22} and its proof.
\end{proof}

\begin{defn}\label{defn:J^AH}
Suppose $\sigma$ is a Serre weight and $\mathcal{S}(\chi_1,\chi_2,\sigma)$ is non-empty with maximal element $(J,x) = (J^\mathrm{max},x^\mathrm{max})$. Let $s = s(J,x)$ as in Defn.~\ref{defn:order-on-S-chi1-chi2-sigma} and $t_\tau = a_\tau - b_\tau + e - s_\tau$ for all $\tau$. For each $\tau\in\Hom_{\FFp}(k,\FFpb)$, we define
\[
\mathcal{I}_\tau:=\begin{cases}
[0,s_\tau - 1] &\text{ if }\tau\notin J; \\
\{t_\tau\} \cup [r_\tau, s_\tau - 1] &\text{ if }\tau\in J.
\end{cases}
\]
Moreover, for each $\tau\in\Hom_{\FFp}(k,\FFpb)$, we define the constants
\[
\xi_\tau:= (p^f-1)s_\tau + \Omega_{\tau,s-t}.
\]
We let $J_\sigma^\mathrm{AH}(\chi_1,\chi_2)$ denote the subset of all $\alpha = (m,k) \in W$ such that there exist $\tau\in \Hom_{\FFp}(k,\FFpb)$, $d\in \mathcal{I}_\tau$ and $j\ge 0$ such that
\begin{enumerate}
\item $p^jm = \xi_\tau - d(p^f-1)$ and
\item $\tau_\alpha = \tau \circ \varphi^j$.
\end{enumerate}
\end{defn}

\begin{rem}\label{rem:only-first-eqn-in-J^AH}
We note that $s_\tau = s_{\tau\circ\varphi^{if'}}$ for any $\tau$ and $0\le i <f''$; hence, similarly for $t$ and $\xi$. Therefore, we see that $(m,k)\in J_\sigma^\mathrm{AH}(\chi_1,\chi_2)$ implies $(m,i)\in J_\sigma^\mathrm{AH}(\chi_1,\chi_2)$ for all $0\le i < f''$.
\end{rem}

\begin{defn}\label{defn:L^AH}
Let $\sigma = \sigma_{a,b}$ be a Serre weight. We define $L_\sigma^\mathrm{AH}(\chi_1,\chi_2)$ to be the span of
\[
\{c_\alpha\mid \alpha \in J_\sigma^\mathrm{AH}(\chi_1,\chi_2)\}
\]
together with $c_\mathrm{un}$ (respectively, $c_\mathrm{tr}$) if $\chi$ is trivial (respectively, if $\chi$ is cyclotomic, $\chi_2$ is unramified and $r_\tau =p$ for all $\tau$).
\end{defn}

This explicit definition of the set of extensions allows us to give an explicit definition of the set of Serre weights associated to $\ovr{r}$ in the non-semisimple case. The following definition is first found in \cite{ddr16} for $K/\QQp$ unramified and in \cite{ste22} for general $K/\QQp$.

\begin{defn}\label{defn:w-exp-reducible}
Suppose that $\ovr{r}\sim\begin{pmatrix} \chi_1 & c_{\ovr{r}} \\ 0 & \chi_2 \end{pmatrix}$. We define an explicit set of associated Serre weights $W^\mathrm{exp}(\ovr{r})$ as follows: $\sigma \in W^\mathrm{exp}(\ovr{r})$ if and only if 
\begin{enumerate}
\item $\sigma\in W^\mathrm{exp}(\ovr{r}^\mathrm{ss})$ and
\item $c_{\ovr{r}} \in L^\mathrm{AH}_\sigma(\chi_1,\chi_2)$.
\end{enumerate}
\end{defn}

The following theorem shows that we have indeed found an equivalent explicit description of the set of weights associated to $\ovr{r}$ when $\ovr{r}$ is reducible. This result is due to \cite{cegm17} in the unramified case and \cite{ste22} in the general case.

\begin{thm}\label{thm:L=L^AH}
Suppose that $\sigma \in W^\mathrm{exp}(\ovr{r}^\mathrm{ss})$. Then $L^\mathrm{AH}_\sigma(\chi_1,\chi_2) = L_\sigma(\chi_1,\chi_2)$.
\end{thm}
\begin{proof}
This follows immediately from Thm.~\ref{thm:GLS-W^cr=W^exp} and \cite[Thm.~4.16]{ste22}.
\end{proof}

\begin{cor}\label{cor:W^cr=W^exp-nonsemisimple}
We have
\[
W^\mathrm{exp}(\ovr{r}) = W^\mathrm{cr}(\ovr{r}).
\]
\end{cor}

\begin{proof}
Immediate consequence of Thm.~\ref{thm:L=L^AH}, Thm.~\ref{thm:GLS-W^cr=W^exp} and Thm.~\ref{thm:W^cr=W^exp-semisimple}.
\end{proof}

\section{The main result}

In this section will state the main local result of this paper. Recall that $\ovr{r}\colon G_K\to \GL_2(\FFpb)$ is a continuous reducible representation of the form
\[
\ovr{r}\sim \begin{pmatrix}\chi_1 & * \\ 0 & \chi_2 \end{pmatrix}
\]
for $\chi_1, \chi_2\colon G_K\to \FFpb^\times$ continuous characters. Denote the quotient by $\chi = \chi_1\chi_2^{-1}$ and write $\chi = \psi\prod_{\tau\in\Hom_{\FFp}(k,\FFpb)}\omega_\tau^{n_\tau}$ for $\psi$ an unramified character and $n_\tau \in [1,p]$ for all $\tau$ with at least one $n_\tau<p$. Note that this uniquely defines $n\in \ZZ^{\Hom_{\FFp}(k,\FFpb)}$.

\begin{hypo}[Generic Hypothesis]\label{hypo:generic}
We will say that $\ovr{r}$ (or $\chi$) is weakly generic if $n_\tau\in[e,p-e]$ for all $\tau$. We will say that $\ovr{r}$ (or $\chi$) is strongly generic if $n_\tau\in[e,p-1-e]$ for all $\tau$.
\end{hypo}

Note that weak genericity implies that $e<p/2$ if $p$ is odd and $e=1$ if $p=2$, whereas strong genericity implies that $e\le(p-1)/2$ (hence, $p>2$).

\begin{lem}\label{lem:m-tau-j}
Suppose $\chi$ is weakly generic. For $\tau\in \Hom_{\FFp}(k,\FFpb)$ and $0\le j <e$, write
\[
m_{\tau,j}:= \Omega_{n,\tau} + j(p^f-1).
\]
Then $W_j' = \left\{m_{\tau,j}\mid \tau\in\Hom_{\FFp}(k,\FFpb)\right\}$.
\end{lem}
\begin{proof}
This is \cite[Lem.~7.4.2]{ste20}.
\end{proof}

\begin{defn}\label{defn:dim-vec-l} Suppose $\sigma = \sigma_{a,b}$ is a Serre weight and $\mathcal{S}(\chi_1,\chi_2,\sigma)$ is non-empty with maximal element $(J,x) = (J^\mathrm{max},x^\mathrm{max})$. We define the dimension vector $\ell = (\ell_\tau)_\tau \in [0,e]^{\Hom_{\FFp}(k,\FFpb)}$ as follows: for $\tau \in \Hom_{\FFp}(k,\FFpb)$, we define
\[
\ell_\tau:=
\begin{cases}
x_\tau		&\text{if $\tau \circ \varphi^{-1}\notin J$;} \\
x_\tau+1	&\text{if $\tau \circ \varphi^{-1}\in J$.}
\end{cases}
\]
\end{defn}

\begin{thm}\label{thm:gen-conj}
Suppose $\sigma$ is a Serre weight. Let $\tau \in \Hom_{\FFp}(k,\FFpb)$, $0\le j <e$ and $0\le k <f''$. If $\chi=\chi_1\chi_2^{-1}$ is weakly generic, then
\[
(m_{\tau,j},k)\in J_\sigma^\mathrm{AH}(\chi_1,\chi_2)\;\text{if and only if}\; j<\ell_\tau.
\]
\end{thm}

We will prove this theorem in \S\ref{sec:proof-of-conj}. Now we will show how to use Theorem \ref{thm:gen-conj} to obtain a decomposition of the set of Serre weights. 

\begin{defn}
We let $\mathcal{W}$ denote the set of tuples $(w_\tau)_\tau\in \ZZ^{\Hom_{\FFp}(k,\FFpb)}$ such that $0\le w_\tau \le e$ for all $\tau$. We impose $\mathcal{W}$ with the usual product partial ordering, i.e. we define $\le$ on $\mathcal{W}$ by
\[
(w_\tau)_\tau \le (w_\tau ')_\tau\text{ if and only if } w_\tau \le w_\tau'\text{ for all $\tau$.}
\]
\end{defn}

\begin{defn}
For $w = (w_\tau)_\tau \in \mathcal{W}$, write
\[
L_w:=\mathrm{Span}(\{c_\alpha\mid \alpha = (m_{\tau,j},k)\text{ with }0\le j < e - w_\tau\})
\]
unless $\chi$ is trivial, in which case we additionally include $c_\mathrm{un}$ in this span. We define a subset of Serre weights $P_w$ via
\[
\sigma \in P_w :\iff \sigma\in W^\mathrm{exp}(\ovr{r}^\mathrm{ss})\text{ and } L_\sigma(\chi_1,\chi_2) = L_w.
\]
\end{defn}

\begin{prop}\label{prop:P-w-size-strong-case}
Suppose $\ovr{r}$ is strongly generic and $w = (w_\tau)_\tau \in \mathcal{W}$. Then we have that $|P_w| = 2^{f-\delta_w}$, where 
\[
\delta_w = |\{\tau\in\Hom_{\FFp}(k,\FFpb)\mid w_\tau = 0\text{ or }w_\tau = e\}|,
\]
unless we are in one of the following exceptional cases. If $\chi$ is cyclotomic and $w_\tau = 0$ for all $\tau$ or $\chi^{-1}$ is cyclotomic and $w_\tau = e$ for all $\tau$, then $|P_w| = 2$.
\end{prop}

This proposition is analogous to \cite[Prop.~3.8(2)]{ds15}, although they use a different definition of the `packets' of weights $P_w$. We will give a proof using some of the combinatorics developed in \S\ref{sec:proof-of-conj}, hence we will defer the proof until then.

\begin{prop}\label{prop:P-w-size-weak-case}
Suppose $\ovr{r}$ is weakly generic, but not strongly generic. Let $w = (w_\tau)_\tau \in \mathcal{W}$. Then we have that $ 0\le |P_w| \le 2^{f-\delta_w}$, where 
\[
\delta_w = |\{\tau\in\Hom_{\FFp}(k,\FFpb)\mid w_\tau = 0\text{ or }w_\tau = e\}|.
\]
\end{prop}

Since the proof of this proposition is similar to the proof of Prop.~\ref{prop:P-w-size-strong-case}, we will also postpone its proof until \S\ref{sec:proof-of-conj}.

\begin{defn}
We will say $\ovr{r}$ is tr\`es ramifi\'ee if $c_{\ovr{r}}\notin L_0$.
\end{defn}

See \cite[Chap.~4]{ste22}, in particular Prop.~4.1.2, for the relationship of this definition to other definitions of tr\`es ramifi\'ee. We remark that $c_{\ovr{r}}\in L_0$ always holds unless $\chi$ is cyclotomic, $\chi_2$ is unramified and $a_\tau - b_\tau = p-1$ for all $\tau$. In the latter case, $L_0\subseteq H^1(G_K,\FFpb(\chi))$ forms a codimension 1 subspace and $\ovr{r}$ is tr\`es ramifi\'ee precisely when it is not contained in this subspace.

\begin{thm}\label{thm:mainthm}
Suppose $\ovr{r}$ is weakly generic. We have
\[
W(\ovr{r}) = \coprod_{w \le w^\mathrm{max}} P_w
\]
for some $w^\mathrm{max} \in \mathcal{W}$ depending on $\ovr{r}$, unless $\ovr{r}$ is tr\`es ramifi\'ee in which case $W(\ovr{r}) = \{\sigma_{a,b}\}$ where $\chi_2|_{I_K} = \prod_{\tau} \omega_\tau^{b_\tau}$ and $a_\tau - b_\tau = p-1$ for all $\tau$. 
\end{thm}
\begin{proof}
Suppose that $\ovr{r}$ is not tr\`es ramifi\'ee. It follows from the inclusion $W(\ovr r)\subset W(\ovr r^\mathrm{ss})$, Thm.~\ref{thm:L=L^AH} and Thm.~\ref{thm:gen-conj} that $W(\ovr{r})\subseteq \coprod_{w \in \mathcal{W}} P_w$. Conversely, let $c_{\ovr{r}}$ denote the extension class associated to $\ovr{r}$ and write $\mathcal{W}_{\ovr{r}}$ for all $w\in\mathcal{W}$ such that $c_{\ovr{r}} \in L_w$. (In other words, all $w\in\mathcal{W}$ such that $P_w\subseteq W(\ovr{r})$.) Since $\ovr{r}$ is not tr\`es ramifi\'ee, $c_{\ovr r}$ lies in $L_0$ so $\mathcal{W}_{\ovr{r}}$ is non-empty. Take $w^\mathrm{max}:=\max_{w\in \mathcal{W}_{\ovr r}}\{w\}$. Then 
\[
L_{w^\mathrm{max}} = \bigcap_{w\in \mathcal{W}_{\ovr r}} L_w.
\]
Hence, $c_{\ovr r} \in L_{w^\mathrm{max}}$. Furthermore, $w\le w^\mathrm{max}$ implies $c_{\ovr r} \in L_w$. It follows that
\[
W(\ovr{r}) = \coprod_{w \le w^\mathrm{max}} P_w.
\]

Now suppose that $\ovr{r}$ is tr\`es ramifi\'ee. Then $c_{\ovr r}$ does not lie in the codimension one subspace $L_0$ of $H^1(G_K,\FFpb(\chi))$. It follows from \cite[Cor.~6.2]{bs22} that $\sigma_{a,b} \in W(\ovr r)$ with $\sigma_{a,b}$ as in the statement of the theorem. On the other hand, if $\sigma = \sigma_{a,b}$ is a Serre weight with $r_\tau = a_\tau - b_\tau <p-1$ for some $\tau$, then it follows from \cite[Thm.~4.1.1]{ste20} that $c_{\ovr r}\notin L_\sigma(\chi_1,\chi_2)$. Therefore, $\sigma\notin W(\ovr r)$.
\end{proof}

\section{Proof of Theorem~\ref{thm:gen-conj}}\label{sec:proof-of-conj}

In this section we will prove Theorem~\ref{thm:gen-conj} and Propositions~\ref{prop:P-w-size-strong-case} and \ref{prop:P-w-size-weak-case}. Most of the proof of Thm.~\ref{thm:gen-conj} will consist of delicate combinatorial arguments. We will start with a Lemma giving an explicit solution to a congruence.

\subsection{Explicitly solving a congruence}

Suppose $J\subseteq \Hom_{\FFp}(k,\FFpb)$ and $c\in \ZZ^{\Hom_{\FFp}(k,\FFpb)}$ such that $c_\tau \in[1,p-1]$ for all $\tau$. Recall the definition
\[
(-1)^{\tau \notin J}:=\begin{cases}-1 &\text{if }\tau\notin J\text{;} \\1 &\text{if }\tau \in J.  \end{cases}
\]
In this subsection we would like to study solutions $r\in\ZZ^{\Hom_{\FFp}(k,\FFpb)}$ to the congruence
\[
\sum_{i=0}^{f-1} (-1)^{\tau \circ \varphi^i \notin J} r_{\tau \circ \varphi^i}p^i \equiv \Omega_{\tau,c} \pmod{p^f-1}
\]
satisfying $r_\tau \in [1,p]$ for all $\tau$. Such a solution always exists, but need not be unique for a given $(J,c)$. Let us give an explicit definition of a solution $r$ to this congruence.

\begin{rem}
In \cite{ds15} the authors restrict themselves to strongly generic representations for which the above congruence needs to be solved only when $c_\tau \in [1,p-2]$ for all $\tau$. This is a far easier combinatorial problem and may have been one of the reasons why the authors adopted this more restrictive definition of genericity.
\end{rem}

\begin{defn}\label{defn:delta-J}
Let $\delta_J\colon \ZZ^{\Hom_{\FFp}(k,\FFpb)}\times \Hom_{\FFp}(k,\FFpb) \to \ZZ^{\Hom_{\FFp}(k,\FFpb)}$ be defined as follows. Fix $\tau\in\Hom_{\FFp}(k,\FFpb)$ and let us write $(y_\sigma)_{\sigma\in \Hom_{\FFp}(k,\FFpb)}= \delta_J((x_\sigma)_\sigma,\tau)$. If $1\le x_\tau \le p$, then we set $y_\sigma := x_\sigma$ for all $\sigma \in \Hom_{\FFp}(k,\FFpb)$. Otherwise, we have two special cases.

\begin{enumerate}
\item\label{delta-case1} If $x_\tau \le 0$, then we set $y_\tau:=x_\tau+p$ and we define
\[
y_{\tau\circ \varphi}:=
\begin{cases}
x_{\tau \circ \varphi}-1 	&\text{if }\tau\circ\varphi\notin J; \\ 
x_{\tau \circ \varphi}+1	&\text{if }\tau\circ\varphi\in J. 
\end{cases}
\]
We define $y_\sigma:=x_\sigma$ for all remaining embeddings $\sigma\in\Hom_{\FFp}(k,\FFpb)\setminus\{\tau,\tau\circ\varphi\}$.

\item\label{delta-case2} If $x_\tau>p$, then we set $y_\tau:=x_\tau-p$ and we define
\[
y_{\tau\circ\varphi}:=
\begin{cases}  
x_{\tau\circ\varphi}-1	&\text{if }\tau\circ\varphi\notin J;\\
x_{\tau\circ\varphi}+1 &\text{if }\tau\circ\varphi\in J.
\end{cases}
\]
We define $y_\sigma:=x_\sigma$ for all remaining embeddings $\sigma\in\Hom_{\FFp}(k,\FFpb)\setminus\{\tau,\tau\circ\varphi\}$.
\end{enumerate} 
\end{defn}

\begin{defn}\label{defn:r-J-c}
Suppose $J\subseteq \Hom_{\FFp}(k,\FFpb)$ and $c\in \ZZ^{\Hom_{\FFp}(k,\FFpb)}$ such that $c_\tau \in[1,p-1]$ for all $\tau$. We define $r(J,c)\in \ZZ^{\Hom_{\FFp}(k,\FFpb)}$ recursively by the following process. Define $y_0\in \ZZ^{\Hom_{\FFp}(k,\FFpb)}$ as
\[
y_{0,\tau}:=\begin{cases}
c_\tau 		&\text{ if }\tau \in J, \tau\circ\varphi^{-1}\in J; \\
c_\tau +1 	&\text{ if }\tau \in J, \tau\circ\varphi^{-1}\notin J; \\
p-c_\tau	&\text{ if }\tau\notin J, \tau\circ\varphi^{-1}\in J; \\
p-1-c_\tau	&\text{ if }\tau\notin J, \tau\circ\varphi^{-1}\notin J.  
\end{cases}
\]
Note that $y_{0,\tau}\ge 0$ for all $\tau$. If $y_{0,\tau}>0$ for all $\tau$, then we set $r(J,c) = y_0$. Otherwise, we fix $\tau_0$ such that $y_{0,\tau_0} = 0$. For $\kappa=1,\dotsc,f$, we then define $y_\kappa\in \ZZ^{\Hom_{\FFp}(k,\FFpb)}$ recursively as $y_\kappa = \delta_J(y_{\kappa-1},\tau_0\circ\varphi^{\kappa-1})$. We set $r(J,c) = y_f$.
\end{defn}

\begin{rem}
For notational purposes, when $y_{0,\tau}>0$ for all $\tau$ we may assume when necessary that $y_\kappa = y_0$ for all $\kappa=1,\dotsc,f$.
\end{rem}

\begin{rem}
We will prove in Prop.~\ref{prop:generic-implies-unique-weight} that $r(J,c)$ is independent of the choice of $\tau_0$ in its definition in the context in which we will need it.
\end{rem}

\begin{lem}\label{lem:r-J-c-solves-congruence}
Suppose $J\subseteq \Hom_{\FFp}(k,\FFpb)$ and $c\in \ZZ^{\Hom_{\FFp}(k,\FFpb)}$ such that $c_\tau \in[1,p-1]$ for all $\tau$. Then $r = r(J,c) \in \ZZ^{\Hom_{\FFp}(k,\FFpb)}$ (defined using any appropriate choice of $\tau_0$ if necessary) solves the congruence
\[
\sum_{i=0}^{f-1} (-1)^{\tau \circ \varphi^i \notin J} r_{\tau \circ \varphi^i}p^i \equiv \Omega_{\tau,c} \pmod{p^f-1}
\]
and satisfies $r_\tau \in [1,p]$ for all $\tau$.
\end{lem}
\begin{proof}
We adopt the same notation as in Defn.~\ref{defn:r-J-c}. It follows immediately from the definitions that $y_0$ solves the congruence and $y_{0,\tau}\in [0,p]$ for all $\tau$. Thus, the lemma is proved if $y_{0,\tau}>0$ for all $\tau$. Therefore, we may assume we have fixed a choice $\tau_0$ such that $y_{0,\tau_0} = 0$. Then $c_{\tau_0} = p-1$ and $\tau_0,\tau_0\circ \varphi^{-1}\notin J$.

For any $\tau \in \Hom_{\FFp}(k,\FFpb)$, define $v^\tau\in\ZZ^{\Hom_{\FFp}(k,\FFpb)}$ via $v^\tau_\tau = p$ and 
\[
v^\tau_{\tau\circ \varphi} = \begin{cases}-1 &\text{if $\tau,\tau\circ\varphi \in J$ or $\tau,\tau\circ\varphi \notin J$;} \\ 1 &\text{if $\tau\in J$, $\tau\circ\varphi\notin J$ or $\tau\notin J$, $\tau\circ\varphi\in J$,} 
\end{cases}
\]
and let $v^\tau_{\tau'} = 0$ for all $\tau'\in\Hom_{\FFp}(k,\FFpb)$. We see easily that adding $\pm v^\tau$, for any $\tau$, to $y_0$ does not change the congruence. For $\kappa=0,\dotsc,f-1$, we claim that $y_{\kappa+1}=\delta_J(y_\kappa,\tau_0\circ\varphi^\kappa)$ leaves $y_\kappa$ invariant, adds $v^{\tau_0\circ \varphi^\kappa}$ to $y_\kappa$ or subtracts $v^{\tau_0\circ \varphi^\kappa}$ from $y_\kappa$ implying the congruence remains unchanged. We note that $y_{\kappa,\tau_0\circ\varphi^{\kappa}}\le 0$ implies that $y_{0,\tau_0\circ\varphi^{\kappa}}=0$ or $y_{\kappa,\tau_0\circ\varphi^{\kappa}}=y_{0,\tau_0\circ\varphi^{\kappa}}-1$. Either case implies $\tau_0\circ\varphi^{\kappa}\notin J$. Applying Case~(\ref{delta-case1}) of Defn.~\ref{defn:delta-J}, we see that if $y_{\kappa,\tau_0\circ\varphi^{\kappa}}\le 0$, then $\delta_J(y_\kappa,\tau_0\circ\varphi^\kappa)$ changes $y_\kappa$ by adding $v^{\tau_0\circ \varphi^\kappa}$. On the other hand, $y_{\kappa,\tau_0\circ\varphi^{\kappa}}>p$ implies that $y_{\kappa,\tau_0\circ\varphi^{\kappa}} = y_{0,\tau_0\circ\varphi^{\kappa}}+1$ which only occurs when $\tau_0\circ\varphi^{\kappa}\in J$. Hence, applying Case~(\ref{delta-case2}) of Defn.~\ref{defn:delta-J}, we see that if $y_{\kappa,\tau_0\circ\varphi^{\kappa}}>p$, then $\delta_J(y_\kappa,\tau_0\circ\varphi^\kappa)$ changes $y_\kappa$ by subtracting $v^{\tau_0\circ \varphi^\kappa}$.

Now we will prove that $r_\tau\in[1,p]$ for all $\tau$. For $\kappa = 1,\dotsc,f$, we claim inductively that $y_{\kappa,\tau_0\circ\varphi^i}\in [1,p]$ for all $0\le i < \kappa$. Since $y_{1,\tau_0} = p$, the claim is true for $\kappa = 1$. Suppose it is true for $\kappa = n$. To prove the claim for $\kappa = n+1$, we note that $y_{n+1,\tau_0\circ\varphi^i} = y_{n,\tau_0\circ\varphi^i}$ for all $0\le i <n$, except that, possibly, when $n=f-1$, we have that $y_{f,\tau_0}\neq y_{f-1,\tau_0}$. In the latter case, since $\tau_0\notin J$ and $y_{f-1,\tau_0}=y_{1,\tau_0} = p$, we have that $y_{f,\tau_0}=p-1\in[1,p]$, so we may disregard this case. Otherwise, to complete the proof we must show that $y_{n+1,\tau_0\circ \varphi^{n}}\in[1,p]$. Since $y_{0,\tau_0\circ\varphi^n}\in[0,p]$, it is clear that $y_{n,\tau_0\circ\varphi^n}\in [-1,p+1]$. If $y_{n,\tau_0\circ\varphi^n}\in\{-1,0\}$, then $y_{n+1,\tau_0\circ\varphi^n}\in\{p-1,p\}$. If $y_{n,\tau_0\circ\varphi^n}=p+1$, then $y_{n+1,\tau_0\circ\varphi^n}=1$. If $y_{n,\tau_0\circ\varphi^n}\in [1,p]$, then $y_{n+1,\tau_0\circ\varphi^n}=y_{n,\tau_0\circ\varphi^n}$. In all cases, $y_{n+1,\tau_0\circ\varphi^n}\in [1,p]$ which completes the proof.
\end{proof}

%
%

\subsection{Unique Serre weight in the weakly generic case}

Suppose $\ovr{r}\sim\begin{pmatrix}\chi_1 & * \\ 0 & \chi_2 \end{pmatrix}$ for characters $\chi_1,\chi_2:G_K\to \FFpb^\times$. Recall that $\chi:=\chi_1\chi_2^{-1}$ and $\chi|_{I_K} = \prod_\tau \omega_\tau^{n_\tau}$ for $n\in\ZZ^{\Hom_{\FFp}(k,\FFpb)}$ with $n_\tau\in [1,p]$, for all $\tau$, and $n_\tau<p$ for at least one $\tau$.

\begin{prop}\label{prop:generic-implies-unique-weight}
Suppose $\ovr{r}$ is weakly generic and $\sigma_{a,b}\in W^\mathrm{exp}(\ovr{r}^\mathrm{ss})$. Let $(J,x) = (J_\mathrm{max},x_\mathrm{max})$ be the unique maximal element of $\mathcal{S}(\chi_1,\chi_2,\sigma_{a,b})$. Define $c_\tau := n_\tau +e-1-2x_\tau \in [1,p-1]$ for all $\tau$. Write $r=r(J,c)\in \ZZ^{\Hom_{\FFp}(k,\FFpb)}$ as in Defn.~\ref{defn:r-J-c}.

Then $a_\tau-b_\tau+1 = r_\tau$ for all $\tau$ except when $J=\Hom_{\FFp}(k,\FFpb)$, $n_\tau = e$ and $x_\tau=e-1$ for all $\tau$ or $J=\varnothing$, $n_\tau=p-1-e$ and $x_\tau = 0$ for all $\tau$. In the latter two cases, we have the additional possibility that $a_\tau -b_\tau +1 = p$ for all $\tau$.

Furthermore, $r(J,c)$ is independent of the choice of $\tau_0$ in its definition. 
\end{prop}

\begin{proof}
It follows from the assumption $(J,x)\in\mathcal{S}(\chi_1,\chi_2,\sigma_{a,b})$ and $\chi|_{I_K} = \prod_\tau \omega_\tau^{n_\tau}$ that
\[
\sum_{i=0}^{f-1} (-1)^{\tau\circ\varphi^i\notin J}(a_{\tau\circ\varphi^i} - b_{\tau\circ\varphi^i} +1)p^i \equiv \sum_{i=0}^{f-1} (n_{\tau\circ\varphi^i}+e -1 - 2x_{\tau\circ\varphi^i})p^i\pmod{p^f-1}.
\]
Therefore, if $c_\tau = n_\tau +e -1 -2x_{\tau}$ and $r_\tau = a_\tau - b_\tau +1$ for all $\tau$, then $c_\tau\in [1,p-1]$ and $r_\tau\in [1,p]$ for all $\tau$. Furthermore, $r$ solves the congruence
\begin{equation}\label{eqn:congruence-r-c}
\sum_{i=0}^{f-1} (-1)^{\tau\circ\varphi^i\notin J} r_{\tau\circ\varphi^i}p^i \equiv \Omega_{\tau,c}\pmod{p^f-1}.
\end{equation}
This congruence may not always have a unique solution $r\in\ZZ^{\Hom_{\FFp}(k,\FFpb)}$ satisfying $r_\tau\in [1,p]$ for all $\tau$. When it does have a unique solution it follows from Lem.~\ref{lem:r-J-c-solves-congruence} that $r=r(J,c)$. 

Therefore, suppose we have $r,r'\in\ZZ^{\Hom_{\FFp}(k,\FFpb)}$ satisfying the congruence with $r_\tau,r_\tau'\in [1,p]$ for all  $\tau$. Then
\[
\sum_{i=0}^{f-1}(-1)^{\tau\circ\varphi^i \notin J}(r_{\tau\circ\varphi^i} - r_{\tau\circ\varphi^i}')p^i \equiv 0 \pmod{p^f-1}.
\]
Since $r_\tau - r_\tau' \in [-(p-1),(p-1)]$ for all $\tau$, this implies that either $r = r'$ or, possibly after interchanging $r$ and $r'$, that
\begin{equation}\label{eqn:two-r-tau-solutions}
r_\tau = \begin{cases} p &\text{if }\tau \in J; \\ 1 &\text{if }\tau \notin J,\end{cases}
\hspace{1cm}\text{and}\hspace{1cm}
r_\tau' = \begin{cases} 1 &\text{if }\tau \in J; \\ p &\text{if }\tau \notin J.\end{cases}
\end{equation}
Suppose $J\neq \varnothing, \Hom_{\FFp}(k,\FFpb)$ and 
\[
r_\tau = \begin{cases} 1 &\text{if }\tau \in J; \\ p &\text{if }\tau \notin J.\end{cases}.
\] 
We claim this contradicts the maximality of $(J,x)$. In this case we may fix $\kappa\in J$ such that $\kappa\circ\varphi^{-1}\notin J$. Then define $J'$ by removing $\kappa$ from $J$ and adding $\kappa\circ\varphi^{-1}$. Let $x'=x$. Now define $s_\tau=x_\tau$ if $\tau\notin J$ and $s_\tau=r_\tau + x_\tau$ if $\tau\in J$ and similarly for $s'$ using $(J',x')$ so that $s=s(J,x)$ and $s' = s(J',x')$. Then it follows that $s_\kappa'-s_\kappa = -1$, $s_{\kappa\circ\varphi^{-1}}'-s_{\kappa\circ\varphi^{-1}} = p$ and $s_\tau'=s_\tau$ for all $\tau\in\Hom_{\FFp}(k,\FFpb)\setminus \{\kappa,\kappa\circ\varphi^{-1}\}$. It follows that $\Omega_{\kappa,s'-s} = (p^f-1)$ and $\Omega_{\tau,s'-s}=0$ for all $\tau\neq \kappa$. This contradicts the maximality of $(J,x)$ according to Prop.~\ref{prop:max-s-t}.

We are left with showing the following: if congruence~(\ref{eqn:congruence-r-c}) has multiple solutions, then either we are in one of the exceptional cases of the proposition or $r(J,c)$ produces the first of the two solutions of (\ref{eqn:two-r-tau-solutions}), i.e. the solution corresponding to $(J_\mathrm{max},x_\mathrm{max})$. For any subset $J\subseteq \Hom_{\FFp}(k,\FFpb)$, let $\1_J\colon \Hom_{\FFp}(k,\FFpb)\to \{0,1\}$ denote the characteristic function of $J$ and let $J^\mathrm{c}$ denote the complement of $J$ in $\Hom_{\FFp}(k,\FFpb)$. The above shows congruence~(\ref{eqn:congruence-r-c}) has a unique solution unless
\[
\Omega_{\tau,c} \equiv \sum_{i=0}^{f-1} (\1_J(\tau\circ\varphi^i)p-\1_{J^\mathrm{c}}(\tau\circ\varphi^i))p^i \pmod{p^f-1}.
\]
Suppose this congruence is satisfied. Then
\[
\sum_{i=0}^{f-1} (c_{\tau\circ\varphi^i} - \1_{J}(\tau\circ\varphi^{i-1})+\1_{J^\mathrm{c}}(\tau\circ\varphi^i))p^i\equiv 0 \pmod{p^f-1}.
\]
Since $c_{\tau\circ\varphi^i} - \1_{J}(\tau\circ\varphi^{i-1})+\1_{J^\mathrm{c}}(\tau\circ\varphi^i)\in [0,p]$ for all $i$, it follows that the sum
\[
S:=\sum_{i=0}^{f-1} (c_{\tau\circ\varphi^i} - \1_{J}(\tau\circ\varphi^{i-1})+\1_{J^\mathrm{c}}(\tau\circ\varphi^i))p^i
\]
is equal to $0$ or $p^f-1$ (resp. $0$, $2^f-1$ or $2(2^f-1)$) if $p>2$ (resp. $p=2$).

Firstly, suppose $p>2$. If $S=0$, then $c_{\tau\circ\varphi^i} - \1_{J}(\tau\circ\varphi^{i-1})+\1_{J^\mathrm{c}}(\tau\circ\varphi^i) = 0$ for all $i$. It is clear that this implies $J=\Hom_{\FFp}(k,\FFpb)$, $n_\tau = e$ and $x_\tau = e-1$ for all $\tau$. Now suppose $S=p^f-1$. If the top coefficient in $S$ were $p$, then $S\ge p^f$. If the top coefficient of $S$ were $\le p-2$, then $S\le (p-2)p^{f-1} +p\left(\frac{p^{f-1}-1}{p-1}\right) < p^f-1$. So we find that $c_{\tau\circ\varphi^{f-1}} - \1_{J}(\tau\circ\varphi^{f-2})+\1_{J^\mathrm{c}}(\tau\circ\varphi^{f-1}) = p-1$. Continuing like this we find that $c_{\tau\circ\varphi^i} - \1_{J}(\tau\circ\varphi^{i-1})+\1_{J^\mathrm{c}}(\tau\circ\varphi^i)=p-1$ for all $i$. From this it follows that $y_{0,\tau}>0$ for all $\tau$ in the definition of $r(J,c)$ (see Defn.~\ref{defn:r-J-c}) so that $r(J,c) = y_0$. Moreover, if $\tau\in J$, then $\tau\circ\varphi^{-1}\notin J$. If $J\neq \varnothing$ and $\tau\in J$, then we see that $y_{0,\tau} =r(J,c)_\tau=p$. Hence, $r(J,c)$ gives the solution of (\ref{eqn:two-r-tau-solutions}) corresponding to $(J_\mathrm{max},x_\mathrm{max})$ in this case. On the other hand, if $J=\varnothing$, then it is easily seen that $n_\tau = p-1-e$ and $x_\tau = 0$, for all $\tau$, is the only way to ensure all coefficients equal $p-1$. Note that in either of the two exceptional cases we have that $r(J,c)_\tau = 1$ for all $\tau$, so that the other solution is given by $r_\tau=p$ for all $\tau$.

Now suppose $p=2$, which implies $e=1$, $n_\tau = 1$ and $x_\tau = 0$, for all $\tau$. As before, if $S = 0$, then $J=\Hom_{\FFp}(k,\FFpb)$. If $S=2(2^f-1)$, then it is also clear that $J=\varnothing$. In either case it follows from the definition that $r(J,c)_\tau = 1$ for all $\tau$. Hence, the other solution will be given by $r_\tau = 2$ for all $\tau$. Lastly, assume $S = 2^f-1$. As before this implies $1 - \1_{J}(\tau\circ\varphi^{i-1})+\1_{J^\mathrm{c}}(\tau\circ\varphi^i)=1$ for all $i$. Rewriting gives $\1_{J}(\tau)=\1_{J^\mathrm{c}}(\tau\circ\varphi)$, for all $\tau$, so that $J\neq \varnothing, \Hom_{\FFp}(k,\FFpb)$. It follows that in the definition of $r(J,c)$ we have $y_{0,\tau}>0$ for all $\tau$. Hence, if $\tau\in J$, then $r(J,c)_\tau = 2$ so that $r(J,c)$ is the solution of (\ref{eqn:congruence-r-c}) corresponding to $(J_\mathrm{max},x_\mathrm{max})$.

It is easily checked that in the exceptional cases $y_{0,\tau}>0$ for all $\tau$ in the definition of $r(J,c)$ so that this definition is not dependent on any choices, except when $p=2$ and $J=\varnothing$; in the latter case, we have $y_{0,\tau}=0$ for all $\tau$ and independence of a choice of $\tau_0$ follows trivially. In all other cases, we have shown that $r(J,c)$ gives the unique solution of congruence (\ref{eqn:congruence-r-c}) corresponding to the unique maximal choice of $(J,x)$. Therefore, $r(J,c)$ is uniquely defined and independent of any choices also in these cases. 
\end{proof}

As a consequence of this proposition we adopt the following notation.

\begin{defn}\label{defn:r-J-x}
Suppose $\ovr{r}$ is weakly generic and $\sigma_{a,b}\in W^\mathrm{exp}(\ovr{r}^\mathrm{ss})$. Let $(J,x) = (J_\mathrm{max},x_\mathrm{max})$ be the unique maximal element of $\mathcal{S}(\chi_1,\chi_2,\sigma_{a,b})$. Since we will only apply Defn.~\ref{defn:r-J-c} with $c_\tau := n_\tau +e-1-2x_\tau$ for all $\tau$, we will write $r(J,x)$ instead of $r(J,c)$ with $c_\tau := n_\tau +e-1-2x_\tau$ for all $\tau$.
\end{defn}

\subsection{Proving the main result}

Our goal in this subsection will be to prove Theorem~\ref{thm:gen-conj}. Throughout we fix a reducible $\ovr{r}\sim\begin{pmatrix}\chi_1 & * \\ 0 & \chi_2\end{pmatrix}$ and assume $\ovr{r}$ is weakly generic in the sense of Hypothesis~\ref{hypo:generic}. We fix a Serre weight $\sigma = \sigma_{a,b}$ and write $r_\tau = a_\tau - b_\tau +1$. We will often switch between the notations $(J,x)$ and $(s,t)$, where we use the translation $s=s(J,x)$ (see Defn.~\ref{defn:order-on-S-chi1-chi2-sigma}) and $t_\tau = a_\tau - b_\tau +e -s_\tau$ for all $\tau$. Explicitly,
\[
s_\tau = \begin{cases} x_\tau &\text{if $\tau\notin J$;} \\ r_\tau +x_\tau &\text{if $\tau\in J$,}\end{cases} 
\hspace{1cm}
t_\tau = \begin{cases} r_\tau+e-1-x_\tau &\text{if $\tau\notin J$;} \\ e-1-x_\tau &\text{if $\tau\in J$.}\end{cases}
\]
Note that $s,t \in \ZZ^{\Hom_{\FFp}(k,\FFpb)}_{\ge 0}$ with, for each $\tau$, $s_\tau + t_\tau = a_\tau - b_\tau +e$ and either $s_\tau\ge a_\tau - b_\tau + 1$ or $t_\tau\ge a_\tau - b_\tau + 1$. Moreover, $\chi_1|_{I_K} = \prod_\tau \omega_\tau^{s_\tau}$ and $\chi_2|_{I_K} = \prod_\tau \omega_\tau^{t_\tau}$. Therefore, $(s,t)$ is as in \cite[Thm.~5.1.5]{gls15}. We will say $(s,t)$ is maximal if the corresponding $(J,x)$ is maximal and we remark that this is equivalent to the maximality in \cite[\S5.3]{gls15} (cf. \cite[Prop.~6.5]{bs22} and its proof).

Conversely, suppose we are given $(s,t)$ such that $s,t \in \ZZ^{\Hom_{\FFp}(k,\FFpb)}_{\ge 0}$ with, for each $\tau$, $s_\tau + t_\tau = a_\tau - b_\tau +e$ and either $s_\tau\ge a_\tau - b_\tau + 1$ or $t_\tau\ge a_\tau - b_\tau + 1$ and, moreover, such that $\chi_1|_{I_K} = \prod_\tau \omega_\tau^{s_\tau}$ and $\chi_2|_{I_K} = \prod_\tau \omega_\tau^{t_\tau}$. Then we define
\[
J=\{\tau\in\Hom_{\FFp}(k,\FFpb)\mid t_\tau \le e-1\},\hspace{1cm} x_\tau=\begin{cases}s_\tau &\text{if $\tau\notin J$;} \\ s_\tau - r_\tau &\text{if $\tau\in J$,}\end{cases}
\]
for all $\tau$.

\begin{lem}
Given $(s,t)$ suppose that $(J,x)$ is defined as above and suppose $(J,x)$ is maximal as in Prop.~\ref{prop:max-s-t}. We have that
\[
J=\{\tau\in\Hom_{\FFp}(k,\FFpb)\mid t_\tau <r_\tau\}.
\]
\end{lem}
\begin{proof}
Let $J=\{\tau\in\Hom_{\FFp}(k,\FFpb)\mid t_\tau \le e-1\}$ and $I=\{\tau\in\Hom_{\FFp}(k,\FFpb)\mid t_\tau <r_\tau\}$. Since $t_\tau<r_\tau$ implies $t_\tau\le e-1$, we have $I\subseteq J$. We show that $J\not\subseteq I$ contradicts that $(J,x)$ is uniquely maximal. Suppose $\tau\in J\setminus I$. Let $J' = J\setminus\{\tau\}$ and $x_\tau' = x_\tau + r_\tau$. Note that $r_\tau\le t_\tau = e - 1 - x_\tau$ implies $r_\tau+x_\tau \le e-1$. Let $x_\kappa' = x_\kappa$ for all $\kappa\neq \tau$. We find $S(J,x) = s(J',x')$ and $\Omega_{\kappa,s(J',x')-s(J,x)} = 0$ for all $\kappa\in\Hom_{\FFp}(k,\FFpb)$. Therefore, $(J,x)\le(J',x')$ by Prop.~\ref{prop:max-s-t} contradicting the fact that $(J,x)$ is the unique maximal element.
\end{proof}

Note that it follows from this lemma that as long as $(J,x)$ and $(s,t)$ are maximal, the two notations are equivalent. When dealing with $r(J,x)$ we will often refer to the notation of Defn.~\ref{defn:r-J-x} or Defn.~\ref{defn:r-J-c} (such as the $y_{0,\tau}$ appearing there) without mentioning this explicitly. We will denote the standard $p$-adic valuation on $\QQ$ by $v_p\colon \QQ \to \ZZ$. Now we will prove four propositions that will help us to prove Thm.~\ref{thm:gen-conj}.

Throughout this section we will fix a Serre weight $\sigma = \sigma_{a,b}$ and assume $(J,x)$ (equivalently, $(s,t)$) is maximal.

\begin{prop}\label{prop:comb-prop1}
For all $\tau\in \Hom_{\FFp}(k,\FFpb)$, we have that $t_\tau\in\mathcal{I}_\tau\text{ if and only if }t_\tau<r_\tau$.
\end{prop}
\begin{proof}
Fix $\tau\in\Hom_{\FFp}(k,\FFpb)$. If $t_\tau <r_\tau$, then it follows from the definition that $t_\tau\in \mathcal{I}_\tau$. Therefore, suppose that $t_\tau\ge r_\tau$. We will show that $t_\tau\ge s_\tau$ proving that $t_\tau\notin \mathcal{I}_\tau$.

Note that $\tau\notin J$. We want to apply Prop.~\ref{prop:generic-implies-unique-weight}. Let us deal with the exceptional cases first. Since we may exclude $J=\Hom_{\FFp}(k,\FFpb)$, suppose $J=\varnothing$, $n_\kappa = p-1-e$, $x_\kappa = 0$ and $r_\kappa = p$ for all $\kappa$. Then $t_\tau = p-1+e$ and $s_\tau = 0$, so it is clear that $t_\tau\ge s_\tau$.

Therefore, it follows from Prop.~\ref{prop:generic-implies-unique-weight} that $r=r(J,x)$. Since $\tau\notin J$, we see that
\[
y_{0,\tau}=\begin{cases}p-e-n_\tau+1+2x_\tau &\text{if $\tau\circ\varphi^{-1}\in J$;} \\ p-e-n_\tau+2x_\tau &\text{if $\tau\circ\varphi^{-1}\notin J$.}\end{cases}
\]
Since $p-e-n_\tau\ge 0$, we find $y_{0,\tau}\ge 2x_\tau$. Similarly, we see that $y_{0,\tau}\le p-1$. Therefore, $r_\tau = y_{f,\tau} \ge y_{0,\tau} - 1$. We conclude that $r_\tau \ge 2x_\tau - 1$.

We note that $t_\tau\ge s_\tau$ if and only if $r_\tau + e -1 - x_\tau \ge x_\tau$. Rewriting gives $r_\tau \ge 2x_\tau +1 -e$. Therefore, the proposition follows immediately from $r_\tau \ge 2x_\tau - 1$ when $e>1$. When $e=1$, we have $x_\tau = 0$ so that $t_\tau = r_\tau > 0 = s_\tau$.
\end{proof}

\begin{prop}\label{prop:comb-prop2}
For all $\tau\in\Hom_{\FFp}(k,\FFpb)$ and all $c\in \mathcal{I}_\tau$, we have that
\[
v_p(\xi_\tau - c(p^f-1))>0
\]
if and only if either $c=t_\tau$ or $t_\tau = 0$, $s_\tau = p-1+e$ and $r_\tau = p$ for all $\tau$ and $c=p$.
\end{prop}
\begin{proof}
Fix $\tau\in\Hom_{\FFp}(k,\FFpb)$. Since $\xi_\tau\equiv -t_\tau\bmod p$, it follows immediately that $\xi_\tau - t_\tau(p^f-1)\equiv 0 \bmod p$. It is also clear that $t_\tau = 0$ and $c=p$ implies that $\xi_\tau - c(p^f-1)\equiv 0 \bmod p$. This finishes the if-direction of the proposition.

For the other direction, let us treat the cases $t_\tau \ge r_\tau$ and $t_\tau <r_\tau$ separately. Suppose $t_\tau \ge r_\tau$. Then $t_\tau>0$ and $t_\tau\notin \mathcal{I}_\tau$ by Prop.~\ref{prop:comb-prop1}. Therefore, we need to show $v_p(\xi_\tau - c(p^f-1)) = 0$ for all $c\in \mathcal{I}_\tau$. Note that $\mathcal{I}_\tau = [0,s_\tau -1 ] = [0,x_\tau - 1]$, so we may assume without loss of generality that $x_\tau>0$. Together with $\tau\notin J$ this excludes both exceptional cases from Prop.~\ref{prop:generic-implies-unique-weight} so that $r = r(J,x)$. Since $\tau\notin J$, we find that $\xi_\tau - c(p^f-1)\equiv -(r_\tau - x_\tau+e-1-c)\bmod{p}$. In the proof of Prop.~\ref{prop:comb-prop1} we have proved already that $r_\tau \ge 2x_\tau + 1-e$. Therefore, $r_\tau - x_\tau+e-1-c\ge x_\tau-c\ge 1$ since $c\le x_\tau -1$. To bound $r_\tau$ from above, we use $r=r(J,x)$ such that
\[
y_{0,\tau}=\begin{cases}p-e-n_\tau+1+2x_\tau &\text{if $\tau\circ\varphi^{-1}\in J$;} \\ p-e-n_\tau+2x_\tau &\text{if $\tau\circ\varphi^{-1}\notin J$.}\end{cases}
\]
Since $x_\tau>0$, it follows that $y_{0,\tau}\in [2,p-1]$. Since $\tau\notin J$, Defn.~\ref{defn:r-J-c} gives $r_\tau\in\{y_{0,\tau},y_{0,\tau}-1\}$. Therefore, $r_\tau - x_\tau + e - 1 - c \le y_{0,\tau} - x_\tau + e -1 -c \le p-n_\tau+x_\tau \le p-1$. Therefore, $\xi_\tau - c(p^f-1)\not\equiv 0 \bmod{p}$.

Now suppose $t_\tau<s_\tau$. Looking at the exceptional cases of Prop.~\ref{prop:generic-implies-unique-weight}, we may exclude $J=\varnothing$ instantly. On the other hand, if $J = \Hom_{\FFp}(k,\FFpb)$, $n_\tau = e$ and $x_\tau = e-1$ for all $\tau$, then we should deal with the case $r_\tau = p$ for all $\tau$ before we can apply the algorithm of Defn.~\ref{defn:r-J-c}. In the latter case, we have $t_\tau = 0$, $s_\tau = p-1+e$ for all $\tau$ and $\mathcal{I}_\tau \setminus\{t_\tau\} = [p, p+e-2]$. Since $e\le p/2$, it is clear that only $c=p$ in this range gives $\xi_\tau - c(p^f-1) \equiv c \equiv 0\bmod{p}$. However, $c=p$ and $t_\tau = 0$, $s_\tau = p-1+e$, $r_\tau = p$ for all $\tau$ was explicitly stated in the proposition as giving $v_p(\xi_\tau - c(p^f-1))>0$.

Having excluded the exceptional cases of Prop.~\ref{prop:generic-implies-unique-weight}, we may now assume $r=r(J,x)$. Let $c\in \mathcal{I}_\tau\setminus\{t_\tau\}$. Since $\tau \in J$, we have $\xi_\tau-c(p^f-1)\equiv x_\tau - e +1 +c\bmod{p}$. We need to show that $x_\tau-e+1+c \not\equiv 0\bmod{p}$. Since $\mathcal{I}_\tau\setminus\{t_\tau\} = [r_\tau,r_\tau + x_\tau - 1]$, we may assume $x_\tau>0$. Using Defn.~\ref{defn:r-J-c}, we find
\[
y_{0,\tau}=\begin{cases}n_\tau+e-1-2x_\tau &\text{if $\tau\circ\varphi^{-1}\in J$;} \\n_\tau+e-2x_\tau &\text{if $\tau\circ\varphi^{-1}\notin J$.} \end{cases}
\]
Since $y_{0,\tau}\in[1,p-2]$ and $\tau\in J$, it follows from Defn.~\ref{defn:r-J-c} that $r_\tau \in\{y_{0,\tau},y_{0,\tau}+1\}$. Therefore, $x_\tau-e+1+c\ge x_\tau -e+1+y_{0,\tau}\ge n_\tau - x_\tau\ge 1$. On the other hand, $x_\tau-e+1+c \le 2x_\tau - e + 1 + y_{0,\tau} \le p-e+1$. Since $e=1$ gives $x_\tau = 0$, we may assume $e>1$. So we have proved $x_\tau-e+1+c\in [1,p-1]$ as required.
\end{proof}

\begin{prop}\label{prop:comb-prop3}
Suppose $\kappa\in\Hom_{\FFp}(k,\FFpb)$ such that $t_\kappa<r_\kappa$. Then
\[
v_p(\xi_\kappa - t_\kappa(p^f-1))>1
\]
if and only if $e=1$ and $r_\tau = p$, $n_\tau = 1$ and $t_\tau = 0$ for all $\tau$.
\end{prop}
\begin{proof}
It if-direction is immediately seen by a direct calculation. For the other direction, first suppose $e=1$ such that $t_\kappa =0$ and $v_p(\xi_\kappa - t_\kappa(p^f-1))=v_p(\xi_\kappa)$. We will use \cite[Lem.~3.6.5]{cegm17} which assumes we are not in the exceptional case $e=1$ and $r_\tau =p$, $n_\tau = 1$ and $t_\tau = 0$ for all $\tau$. It follows from weak genericity that $\Omega_{\tau,n}\in[\frac{p^f-1}{p-1},p^f-1]$ for all $\tau$. Moreover, \cite[Thm.~4.16]{ste22} implies $\xi_\kappa>0$ (cf. remark after \cite[Prop.~4.13]{ste22}). If $m:=v_p(\xi_\kappa)>1$, then \cite[Lem.~3.6.5]{cegm17} gives that $\xi_\kappa = p^{m}(\Omega_{\kappa\circ\varphi^m,n}-(p^f-1))$. So, by genericity $\xi_\kappa\le 0$ contradicting $\xi_\kappa>0$. Hence, $m = 1$ and we assume in the remainder $e>1$.

Since $\xi_\kappa - t_\kappa(p^f-1) \equiv p(s_{\kappa\circ\varphi}-t_{\kappa\circ\varphi})\bmod{p^2}$, we consider $s_{\kappa\circ\varphi}-t_{\kappa\circ\varphi}\bmod{p}$. Firstly, suppose $\kappa\circ\varphi\notin J$. Then the exceptional cases of Prop.~\ref{prop:generic-implies-unique-weight} are excluded, so we may assume $r=r(J,x)$. We find $y_{0,\kappa\circ\varphi} = p-n_{\kappa\circ\varphi}-e+1+2x_{\kappa\circ\varphi}\in [1,p-1]$. Therefore, $r_{\kappa\circ\varphi}\in\{y_{0,\kappa\circ\varphi}, y_{0,\kappa\circ\varphi}-1, y_{0,\kappa\circ\varphi}+p-1\}$. Note that $s_{\kappa\circ\varphi}-t_{\kappa\circ\varphi} = 2x_{\kappa\circ\varphi} - e + 1 -r_{\kappa\circ\varphi}$. If $r_{\kappa\circ\varphi}\in\{y_{0,\kappa\circ\varphi},y_{0,\kappa\circ\varphi}-1\}$, then $2x_{\kappa\circ\varphi} - e + 1 - r_{\kappa\circ\varphi}\in [e-p,1-e]$. Since $e>1$ and $e\le p/2$ by genericity, it follows that $s_{\kappa\circ\varphi}-t_{\kappa\circ\varphi}\not \equiv 0\bmod{p}$; note that $r_{\kappa\circ\varphi}=y_{0,\kappa\circ\varphi}+p-1$ would give the same result modulo $p$ as $r_{\kappa\circ\varphi}=y_{0,\kappa\circ\varphi}-1$.

Now suppose $\kappa\circ\varphi\in J$ so that $s_{\kappa\circ\varphi}-t_{\kappa\circ\varphi} = r_{\kappa\circ\varphi}+2x_{\kappa\circ\varphi}-e+1$. Firstly, we consider the exceptional case $J=\Hom_{\FFp}(k,\FFpb)$ and $n_\tau = e$, $x_\tau = e-1$ and $r_\tau = p$ for all $\tau$ from Prop.~\ref{prop:generic-implies-unique-weight}. Then $r_{\kappa\circ\varphi}+2x_{\kappa\circ\varphi}-e+1 = p+e-1\not\equiv 0 \bmod{p}$ since $e\in [2,p/2]$. Therefore, by Prop.~\ref{prop:generic-implies-unique-weight}, we may assume $r=r(J,x)$. We find that $y_{0,\kappa\circ\varphi} = n_{\kappa\circ\varphi} + e -1 - 2x_{\kappa\circ\varphi}\in[1,p-1]$. Since $\kappa\circ\varphi\in J$, it follows that $r_{\kappa\circ\varphi}\in\{y_{0,\kappa\circ\varphi},y_{0,\kappa\circ\varphi}+1\}$. Therefore, $r_{\kappa\circ\varphi}+2x_{\kappa\circ\varphi}-e+1 \le y_{0,\kappa\circ\varphi}+2x_{\kappa\circ\varphi}-e+2\le p-e+1$ and, similarly, $r_{\kappa\circ\varphi}+2x_{\kappa\circ\varphi}-e+1 \ge 1$. Thus, $s_{\kappa\circ\varphi}-t_{\kappa\circ\varphi}\in [1,p-1]$ since $e>1$.
\end{proof}

\begin{prop}\label{prop:comb-prop4}
Suppose $\sigma$ is a Serre weight. For all $\tau\in\Hom_{\FFp}(k,\FFpb)$, $1\le j <e$ and $0\le k < f''$, we have that
\[
(m_{\tau,j},k)\in J^\mathrm{AH}_\sigma(\chi_1,\chi_2)\text{ implies that }(m_{\tau,j-1},k)\in J^\mathrm{AH}_\sigma(\chi_1,\chi_2).
\]
\end{prop}
\begin{proof}
By Rem.~\ref{rem:only-first-eqn-in-J^AH} we need to show that if we can find a solution to the first equation of Defn.~\ref{defn:J^AH} for $m_{\tau,j}$, then we can find a solution for $m_{\tau,j-1}$ as well. Note that if $r_\tau = p$, $n_\tau = e$ and $t_\tau = 0$ for all $\tau$, then $J^\mathrm{AH}_\sigma(\chi_1,\chi_2) = W$ and the statement is trivially true. Therefore, we may assume we are not in this exceptional case, which we will refer to as the `cyclotomic exceptional case' below. Then it is a consequence of \cite[Thm.~4.16]{ste22} (cf. remark after \cite[Prop.~4.13]{ste22}) that for all $\kappa\in \Hom_{\FFp}(k,\FFpb)$ and $c\in\mathcal{I}_\kappa$ we have that
\[
0<\frac{\xi_\kappa-c(p^f-1)}{p^m}<\frac{ep}{p-1}(p^f-1),
\]
where $m:=v_p(\xi_\kappa-c(p^f-1))$. Furthermore, it follows from weak genericity that for all $\tau\in\Hom_{\FFp}(k,\FFpb)$ we have $\Omega_{\tau,n}\in [\frac{e(p^f-1)}{p-1},(p-e)\frac{p^f-1}{p-1}]$. We will use these facts in the proof below.

Let $\kappa\in\Hom_{\FFp}(k,\FFpb)$. Firstly, suppose that $t_\kappa<r_\kappa$ and $c=t_\kappa$. It follows from Prop.~\ref{prop:comb-prop2} and \ref{prop:comb-prop3} that $v_p(\xi_\kappa - t_\kappa(p^f-1))=1$ and from \cite[Lem.~4.5]{ste22} that $\frac{\xi_\kappa-t_\kappa(p^f-1)}{p} \equiv \Omega_{\kappa\circ\varphi, n}\bmod{(p^f-1)}$. Since
\[
\frac{\xi_\kappa-t_\kappa(p^f-1)}{p} = \sum_{i=0}^{f-1} (s_{\kappa\circ\varphi^{i+1}}-t_{\kappa\circ\varphi^{i+1}})p^i
\]
and $s_{\tau}-t_{\tau}\le p + e -1$ for all $\tau$ with a strict inequality for at least one $\tau$ by exclusion of the cyclotomic exceptional case, we find that $\frac{\xi_\kappa-t_\kappa(p^f-1)}{p}<(p+e-1)\frac{p^f-1}{p-1} = \left(1 + \frac{e}{p-1}\right)(p^f-1)$. Since $\frac{p-e}{p-1} - 1\le 0$ and $\frac{\xi_\kappa-t_\kappa(p^f-1)}{p}>0$, the bounds on $\Omega_{\kappa\circ\varphi, n}$ and the congruence above imply $\frac{\xi_\kappa-t_\kappa(p^f-1)}{p}=\Omega_{\kappa\circ\varphi, n} = m_{\kappa\circ\varphi,0}$.

Now we go back to the general case (i.e. allowing either $t_\kappa<r_\kappa$ or $t_\kappa\ge r_\kappa$). Again using Prop.~\ref{prop:comb-prop2} and \cite[Lem.~4.5]{ste22}, we see $v_p(\xi_\kappa - c(p^f-1)) = 0$ and
\begin{equation}\label{eqn:congr-kappainJ-cnottk}
\xi_\kappa-c(p^f-1) \equiv \Omega_{\kappa,n}\bmod{(p^f-1)},
\end{equation}
for all $c\in\mathcal{I}_\kappa\setminus\{t_\kappa\}$ -- note that the exceptional case from Prop~\ref{prop:comb-prop2} is exactly the (excluded) cyclotomic exceptional case above. It follows from Prop.~\ref{prop:comb-prop1} that 
\[
\mathcal{I}_\kappa\setminus\{t_\kappa\} = \begin{cases} [0,s_\kappa - 1] &\text{if $t_\kappa\ge r_\kappa$;} \\ [r_\kappa,s_\kappa -1] &\text{if $t_\kappa<r_\kappa$.}\end{cases}
\]
Since these sets are otherwise empty, we may assume without loss of generality that $s_\kappa>0$ if $t_\kappa\ge r_\kappa$ and $s_\kappa>r_\kappa$ if $t_\kappa<r_\kappa$. By definition of the integers $m_{\tau,j}$ in Lem.~\ref{lem:m-tau-j}, it now suffices to show 
\[
\xi_\kappa -(s_\kappa -1)(p^f-1) = \begin{cases} m_{\kappa,0} =  \Omega_{\kappa,n}  &\text{if $t_{\kappa\circ\varphi^{-1}}\ge r_{\kappa\circ\varphi^{-1}}$;} \\ m_{\kappa,1} = \Omega_{\kappa,n}+(p^f-1) &\text{if $t_{\kappa\circ\varphi^{-1}}<r_{\kappa\circ\varphi^{-1}}$.}\end{cases}
\]
Note that
\[
\xi_\kappa - (s_\kappa - 1)(p^f-1) = (p^f-1)+\Omega_{\kappa,s-t} = (p^f-1) + \sum_{i=0}^{f-1} (s_{\kappa\circ\varphi^i} - t_{\kappa\circ\varphi^i})p^i.
\]
and that $\xi_\kappa - (s_\kappa - 1)(p^f-1)>0$ by the first paragraph of the proof.

Firstly, suppose $t_{\kappa\circ\varphi^{-1}}\ge r_{\kappa\circ\varphi^{-1}}$. By Congruence~(\ref{eqn:congr-kappainJ-cnottk}), it suffices to show $\xi_\kappa - (s_\kappa - 1)(p^f-1)<\left(1 + \frac{e}{p-1}\right)(p^f-1)$. In other words, we must show $\Omega_{\kappa,s-t}<\frac{e(p^f-1)}{p-1}$. Since $s_\tau - t_\tau \le e-1-r_\tau\le e-2$ if $t_\tau\ge r_\tau$ and $s_\tau - t_\tau \le e-1+r_\tau\le e-1+p$ if $t_\tau<r_\tau$, we find
\[
\Omega_{\kappa,s-t}\le (e-1)\frac{p^f-1}{p-1} - p^{f-1} + p\frac{p^{f-1}-1}{p-1},
\]
which satisfies
\[
(e-1)\frac{p^f-1}{p-1} - p^{f-1} + p\frac{p^{f-1}-1}{p-1} = (e-1)\frac{p^f-1}{p-1} + p\frac{p^{f-2} - 1}{p-1} < \frac{e(p^f-1)}{p-1},
\]
as required.

On the other hand, suppose $t_{\kappa\circ\varphi^{-1}}< r_{\kappa\circ\varphi^{-1}}$. By Congruence~(\ref{eqn:congr-kappainJ-cnottk}) it suffices to show
\[
(p-e)\frac{p^f-1}{p-1} < \xi_\kappa - (s_\kappa - 1)(p^f-1) < \left(2 + \frac{e}{p-1}\right)(p^f-1).
\]
Since we have already shown that $\frac{\xi_{\kappa\circ\varphi^{-1}} - t_{\kappa\circ\varphi^{-1}}(p^f-1)}{p} = \Omega_{\kappa,n}$, it follows from \cite[Thm.~4.16]{ste22} (cf. remark after \cite[Prop.~4.13]{ste22}) that $\xi_\kappa - (s_\kappa - 1)(p^f-1)\neq \Omega_{\kappa,n}$ since otherwise $|J_\sigma^\mathrm{AH}(\chi_1,\chi_2)| < \sum_\tau |\mathcal{I}_\tau|$. Therefore, it suffices to prove the upper bound on $\xi_\kappa - (s_\kappa - 1)(p^f-1)$. Note that $s_\tau - t_\tau \le e-1 +r_\tau - 2t_\tau \le p+e-1$ for all $\tau$ with a strict inequality for at least one $\tau$ by exclusion of the cyclotomic exceptional case. We find that

\begin{align*}
\xi_\kappa - (s_\kappa - 1)(p^f-1) &= (p^f-1) + \Omega_{\kappa,n} \\
&< (p^f - 1) + (p+e-1)\left(\frac{p^f-1}{p-1}\right) \\
& = \left(2 + \frac{e}{p-1}\right)(p^f-1),
\end{align*}
as required.
\end{proof}

Now we can prove the main theorem.

\begin{proof}[Proof of Theorem~\ref{thm:gen-conj}]
If $r_\tau = p$, $n_\tau = e$ and $t_\tau = 0$ for all $\tau$, then $J^\mathrm{AH}_\sigma(\chi_1,\chi_2) = W$ and the theorem is true. Therefore, we may assume we are not in this `cyclotomic exceptional case' below. Note that this excludes the exceptional cases stated in Prop.~\ref{prop:comb-prop2} and \ref{prop:comb-prop3}.

It is a consequence of \cite[Thm.~4.16]{ste22} (cf. remark after \cite[Prop.~4.13]{ste22}) that $|J_\sigma^\mathrm{AH}(\chi_1,\chi_2)| = \sum_\tau |\mathcal{I}_\tau|$. Therefore, we have an injective map
\[
\iota\colon \{(\tau,c)\mid \tau\in\Hom_{\FFp}(k,\FFpb)\text{ and }d\in\mathcal{I}_\tau\}\to W,
\]
which sends $(\tau,c)$ to the corresponding $\alpha=(m,k)\in W$ satisfying the two conditions of Defn.~\ref{defn:J^AH}. Of course, the image of $\iota$ is, by definition, exactly equal to $J_\sigma^\mathrm{AH}(\chi_1,\chi_2)$. Let us study the preimage
\[
\iota^{-1}\left(\{(m_{\tau,j},k) \mid 0\le j<e\text{ and }0\le k<f''\}\right)
\]
for a fixed $\tau\in\Hom_{\FFp}(k,\FFpb)$.

Suppose $\kappa\in\Hom_{\FFp}(k,\FFpb)$ and $c\in\mathcal{I}_\kappa$. By weak genericity, $p\nmid m_{\tau,j}$ for all $\tau\in\Hom_{\FFp}(k,\FFpb)$ and $0\le j <e$. Therefore, the integer $m\ge 0$ in the first condition of Defn.~\ref{defn:J^AH} is uniquely determined and must equal $v_p(\xi_\kappa - c(p^f-1))$. It follows from Prop.~\ref{prop:comb-prop1} and \ref{prop:comb-prop2} that $v_p(\xi_\kappa - c(p^f-1)) = 0$ if either $t_\kappa\ge r_\kappa$ or $t_\kappa<r_\kappa$ and $c\in\mathcal{I}_\kappa\setminus\{t_\kappa\}$. Since $\xi_\kappa - c(p^f-1) \equiv \Omega_{\kappa,n}\bmod{(p^f-1)}$ by \cite[Lem.~4.5]{ste22}, we must have
\[
(\kappa,c)\in \iota^{-1}\left(\{(m_{\kappa,j},k) \mid 0\le j<e\text{ and }0\le k<f''\}\right)
\]
in either of these cases. Moreover, Prop.~\ref{prop:comb-prop3} implies that $v_p(\xi_\kappa - t_\kappa(p^f-1)) = 1$ if $t_\kappa<r_\kappa$. Since in this case $\frac{\xi_\kappa - t_\kappa(p^f-1)}{p} \equiv \Omega_{\kappa\circ\varphi,n}\bmod{(p^f-1)}$ by \cite[Lem.~4.5]{ste22}, we must have
\[
(\kappa,t_\kappa)\in \iota^{-1}\left(\{(m_{\kappa\circ\varphi,j},k) \mid 0\le j<e\text{ and }0\le k<f''\}\right).
\]
Since $\iota$ is injective, we conclude that
\[
|\iota(\iota^{-1}\left(\{(m_{\tau,j},k) \mid 0\le j<e\text{ and }0\le k<f''\}\right))| = 
\begin{cases}
|\mathcal{I}_\tau| &\text{if $t_\tau\ge r_\tau$ and $t_{\tau\circ\varphi^{-1}}\ge r_{\tau\circ\varphi^{-1}}$;} \\

|\mathcal{I}_\tau|+1 &\text{if $t_\tau\ge r_\tau$ and $t_{\tau\circ\varphi^{-1}}< r_{\tau\circ\varphi^{-1}}$;} \\

|\mathcal{I}_\tau|-1 &\text{if $t_\tau< r_\tau$ and $t_{\tau\circ\varphi^{-1}}\ge r_{\tau\circ\varphi^{-1}}$;} \\

|\mathcal{I}_\tau|&\text{if $t_\tau< r_\tau$ and $t_{\tau\circ\varphi^{-1}}< r_{\tau\circ\varphi^{-1}}$.}
\end{cases}
\]
Now it follows from the definitions of $\ell_\tau$ and $\mathcal{I}_\tau$ that $|\iota(\iota^{-1}\left(\{(m_{\tau,j},k) \mid 0\le j<e\text{ and }0\le k<f''\}\right))| = \ell_\tau$ for all $\tau\in\Hom_{\FFp}(k\FFpb)$. Then it follows from Prop.~\ref{prop:comb-prop4} that
\[
(m_{\tau,j},k)\in J_\sigma^\mathrm{AH}(\chi_1,\chi_2)\text{ if and only if }j<\ell_\tau
\]
for all $\tau\in\Hom_{\FFp}(k,\FFpb)$. This proves Theorem~\ref{thm:mainthm}.
\end{proof}

\subsection{Size of the weight packets \texorpdfstring{$P_w$}{Pw}} In this subsection we will proof Prop.~\ref{prop:P-w-size-strong-case} and \ref{prop:P-w-size-weak-case}. The proof of the strongly generic case also follows from the results of \cite{ds15}. We give a proof using Prop.~\ref{prop:generic-implies-unique-weight}.

\begin{proof}[Proof of Prop.~\ref{prop:P-w-size-strong-case}]
Suppose $\ovr{r}$ is strongly generic. If $\sigma_{a,b}\in W^\mathrm{exp}(\ovr{r}^\mathrm{ss})$, then there exists a maximal $(J,x)$, where $J\subset \Hom_{\FFp}(k,\FFpb)$ and $x_\tau\in [0,e-1]$ for all $\tau\in \Hom_{\FFp}(k,\FFpb)$, such that
\[
\ovr{r}^\mathrm{ss}|_{I_K} \cong \begin{pmatrix}
\prod_{\tau\in J} \omega_\tau^{a_\tau + 1 + x_\tau}\prod_{\tau\notin J}\omega_\tau^{b_\tau + x_\tau} & 0 \\
0 & \prod_{\tau\notin J} \omega_\tau^{a_\tau+e-x_\tau}\prod_{\tau \in J} \omega_\tau^{b_\tau + e -1 -x_\tau}
\end{pmatrix}.
\]
If we write $r=r(J,x)$, then Prop.~\ref{prop:generic-implies-unique-weight} implies $a_\tau - b_\tau + 1 = r_\tau$. Since $c_\tau = n_\tau + e-1 - 2 x_\tau \in [1,p-2]$ for all $\tau$, this implies $\sigma_{a,b}$ is explicitly given by
\begin{align*}
b_\tau &= n_{2,\tau} + p - e + x_\tau &a_\tau - b_\tau &= n_\tau + e - 2 - 2x_\tau &\text{ if $\tau\in J$ and $\tau\circ\varphi^{-1}\in J$;} \\
b_\tau &= n_{2,\tau} + p - e + x_\tau &a_\tau - b_\tau &= n_\tau + e - 1 - 2x_\tau &\text{ if $\tau\in J$ and $\tau\circ\varphi^{-1}\notin J$;} \\
b_\tau &= n_{2,\tau} +n_\tau -1 -x_\tau &a_\tau - b_\tau &= p - n_\tau - e + 2x_\tau &\text{ if $\tau\notin J$ and $\tau\circ\varphi^{-1}\in J$;} \\
b_\tau &= n_{2,\tau} +n_\tau -x_\tau &a_\tau - b_\tau &= p - 1 - n_\tau - e + 2x_\tau &\text{ if $\tau\notin J$ and $\tau\circ\varphi^{-1}\notin J$,}
\end{align*}
where $\chi_2|_{I_K}=\prod_\tau \omega_\tau^{n_{2,\tau}}$, unless we are in one of the exceptional cases of Prop.~\ref{prop:generic-implies-unique-weight}. We denote the Serre weight defined by these formulae by $\sigma(J,x)$. We note that $a_\tau - b_\tau < p-1$ in all cases. Therefore, it follows from the proof of Prop.~\ref{prop:generic-implies-unique-weight} that the maximality assumption on $(J,x)$ was not used. Hence, for any non-exceptional pair $(J,x)$ we can find a corresponding $\sigma(J,x)\in W^\mathrm{exp}(\ovr{r}^\mathrm{ss})$ given by the above formulae, where we may drop the assumption that $(J,x)$ is maximal. (In other words, in the strongly generic case $(J,x)$ will always be maximal for $\sigma_{a,b}$ as above. This is no longer true in the weakly generic case.)

In the exceptional cases we define an additional weight $\sigma'(J,x)$ as follows. If $J=\Hom_{\FFp}(k,\FFpb)$, $n_\tau = e$ and $x_\tau = e-1$ for all $\tau$, let $\sigma'(J,x) = \sigma_{a,b}$ with $a_\tau - b_\tau = p-1$ and $b_\tau = n_{2,\tau}$ for all $\tau$. If $J = \varnothing$, $n_\tau = p-1-e$ and $x_\tau = 0$, for all $\tau$, let $\sigma'(J,x) = \sigma_{a,b}$ with $a_\tau - b_\tau = p-1$ and $b_\tau = n_{2,\tau} - e$ for all $\tau$. Since $a_\tau - b_\tau <p-1$ for the non-exceptional weights, it follows that $\sigma'(J,x)$ is never isomorphic to a non-exceptional weight. Both exceptional weights occur when $e = (p-1)/2$ and $n_\tau = e$ for all $\tau$, but comparing values of $b_\tau$ shows they are distinct in this case.

We claim that $\sigma(J,x)\cong \sigma(J',x')$ if and only if $J = J'$ and $x = x'$. Without loss of generality we may assume $n_{2,\tau} = 0$ for all $\tau$. Write $\sigma(J,x) = \sigma_{a,b}$ and $\sigma(J',x') = \sigma_{a',b'}$. We see that $0\le b_\tau \le p-1$ for all $\tau$ with $b_\tau >0$ for at least one $\tau$. Since the same holds for $b'$ and $\sum_{i=0}^{f-1} b_{\tau\circ\varphi^i}p^i \equiv \sum_{i=0}^{f-1} b_{\tau\circ\varphi^i}' p^i\bmod{(p^f-1)}$, we must have $b = b'$. If $J' \neq J$, then, possibly after interchanging the roles of $J$ and $J'$, there is a $\tau\in J$ with $\tau\notin J'$. Therefore,
\[
b_\tau = p-e + x_\tau \ge p-e > p-e-1 \ge n_\tau - x_\tau \ge b_\tau',
\]
contradicting $b = b'$. Hence, $J = J'$. Then it follows immediately from $b=b'$ and $J=J'$ that $x = x'$, proving the claim.

Thus, we have proved that there is a bijection
\[
W^\mathrm{exp}(\ovr{r}^\mathrm{ss}) \to \{(J,x) \mid J\subseteq \Hom_{\FFp}(k,\FFpb)\text{ and integers }x_\tau\in [0,e-1]\text{ for $\tau\in \Hom_{\FFp}(k,\FFpb)$}\} 
\]
unless either (resp. both) $\chi^{\pm 1}$ is cyclotomic in which case the map is $2-$to$-1$ for a single weight (resp. precisely two weights). Therefore,
\[
|W^\mathrm{exp}(\ovr{r}^\mathrm{ss})|=\begin{cases} e^f 2^f &\text{ if $\chi^{\pm 1}$ are not cyclotomic;} \\
e^f 2^f + 1 &\text{ if either $\chi^{\pm 1}$ is cyclotomic;} \\
e^f 2^f + 2 &\text{ if both $\chi^{\pm 1}$ are cyclotomic.}
\end{cases}
\]

Now fix $w\in \mathcal{W}$. It follows from Thm.~\ref{thm:gen-conj} that $\sigma(J,x)\in P_w$ if and only if the dimension vector $\ell$ (as in Defn.~\ref{defn:dim-vec-l}) corresponding to $\sigma(J,x)$ satisfies $\ell_\tau = e - w_\tau$ for all $\tau$. Since $\ell$ corresponding to $\sigma(J,x)$ satisfies
\[
\ell_\tau = \begin{cases}
x_\tau 		&\text{ if $\tau\circ\varphi^{-1}\notin J$;} \\
x_\tau + 1 	&\text{ if $\tau\circ\varphi^{-1}\in J$,}
\end{cases}
\]
the only imposed conditions on $J$ are: if $w_\tau = e$, then $\tau\circ\varphi^{-1}\notin J$, and if $w_\tau = 0$, then $\tau\circ\varphi^{-1}\in J$. It is clear that $x$ is uniquely determined by $J$ and $w$. Therefore, we get $2^{f-\delta_w}$ non-exceptional weights in $P_w$.

Considering the exceptional weights, it is easily seen that if $\chi$ is cyclotomic, $J=\Hom_{\FFp}(k,\FFpb)$ and $x_\tau = e-1$ for all $\tau$, then $\sigma'(J,x)\in P_{w}$ for $w_\tau = 0$ for all $\tau$. On the other hand, if $\chi^{-1}$ is cyclotomic, $J=\varnothing$ and $x_\tau = 0$ for all $\tau$, then we see $\sigma'(J,x)\in P_w$ for $w_\tau = e$ for all $\tau$. This completes the proof. 
\end{proof}

\begin{proof}[Proof of Prop.~\ref{prop:P-w-size-weak-case}] Suppose that $\ovr{r}$ is weakly generic, but not strongly generic. In particular, note that this excludes the exceptional cases of Prop.~\ref{prop:generic-implies-unique-weight}.

Suppose $\sigma=\sigma_{a,b}\in W^\mathrm{exp}(\ovr{r}^\mathrm{ss})$. It follows from Defn.~\ref{defn:w-exp-semisimple} and Prop.~\ref{prop:max-s-t} that there exists a unique maximal $(J,x)$ such that
\[
\ovr{r}^\mathrm{ss}|_{I_K} \cong \begin{pmatrix}
\prod_{\tau\in J} \omega_\tau^{a_\tau + 1 + x_\tau}\prod_{\tau\notin J}\omega_\tau^{b_\tau + x_\tau} & 0 \\
0 & \prod_{\tau\notin J} \omega_\tau^{a_\tau+e-x_\tau}\prod_{\tau \in J} \omega_\tau^{b_\tau + e -1 -x_\tau}
\end{pmatrix}.
\]
Moreover, it follows from Prop.~\ref{prop:generic-implies-unique-weight} that $a_\tau - b_\tau + 1 = r_\tau$, for all $\tau$, where $r_\tau = r(J,x)$. We immediately see that $\sum_{i=0}^{f-1}b_{\tau\circ\varphi^i}p^i \bmod{(p^f-1)}$ is uniquely defined given $(J,x)$ and $r_\tau = a_\tau - b_\tau+1$ for all $\tau$. If we write $\chi_2|_{I_K} = \prod_\tau \omega_\tau^{n_{2,\tau}}$, then $b$ may be explicitly given by
\[
b_\tau =\begin{cases}
n_{2,\tau}+x_\tau - e + 1 &\text{if $\tau\in J$;} \\
n_{2,\tau}+x_\tau - e + 1 -r_\tau &\text{if $\tau\notin J$.}
\end{cases}
\]
Hence, $\sigma$ is uniquely defined by $(J,x)$. Therefore, we have just proved that there is an injective map
\[
\vartheta\colon W^\mathrm{exp}(\ovr{r}^\mathrm{ss}) \hookrightarrow \{(J,x) \mid J\subseteq \Hom_{\FFp}(k,\FFpb)\text{ and integers }x_\tau\in [0,e-1]\text{ for $\tau\in \Hom_{\FFp}(k,\FFpb)$}\} 
\]
sending $\sigma$ to the unique maximal $(J,x)$ as above. Therefore, $|W^\mathrm{exp}(\ovr{r}^\mathrm{ss})|\le e^f 2^f$. (The crucial difference with the strongly generic case being that this injective map may not be surjective.)

Now fix $w\in\mathcal{W}$ and suppose $(J,x)\in\mathrm{Im}(\vartheta)$ corresponds with a $\sigma\in P_w$. Exactly as in the proof of Prop.~\ref{prop:P-w-size-strong-case}, we find that if $w_\tau = e$, then $\tau\circ\varphi^{-1}\notin J$, and if $w_\tau = 0$, then $\tau\circ\varphi^{-1}\in J$. As before, we see that $x$ is uniquely determined by $J$ and $w$. Therefore, we get at most $2^{f-\delta_w}$ weights in $P_w$.
\end{proof}

\section{A remark on genericity}\label{sec:rem-on-genericity}

Suppose $\ovr{r}\colon G_K\to \GL_2(\FFpb)$ reducible as before with $\ovr{r}^\mathrm{ss}\cong \chi_1\oplus \chi_2$. Write $\chi_1\chi_2^{-1}|_{I_K} = \prod_\tau \omega_\tau^{n_\tau}$ with $n_\tau\in [1,p]$ for all $\tau$. Thm.~\ref{thm:mainthm} was proved in \cite{ds15} under the assumption that $\ovr{r}$ was strongly generic (see Hypo.~\ref{hypo:generic}). However, calculations in \cite{ste22} suggested to us that Thm.~\ref{thm:gen-conj} should be true under the weaker condition that $\ovr{r}$ was only weakly generic (see Hypo.~\ref{hypo:generic}) which led to this paper. Let us briefly discuss here what evidence there is for these various hypotheses being optimal. 

It follows from \cite[\S5]{ds15} that their techniques no longer apply in the case $f=2$, $e=1$ and $(n_{\tau_0},n_{\tau_1}) = (p-1,b)$ with $b\in[1,p-2]$ for some choice $\{\tau_0,\tau_1\} = \Hom_{\FFp}(k,\FFpb)$. However, the authors there remark that indeed it follows from \cite{cd11} that Thm.~\ref{thm:mainthm} should still hold in this case despite the techniques of \cite{ds15} no longer applying. Since this example is weakly, but not strongly generic, the examples in \cite[\S5]{ds15} therefore show that the methods of \cite{ds15} would not have sufficed in the weakly generic case (even if the authors' of \cite{ds15} had used more complicated combinatorics).

When $e=1$ it follows from \cite[\S7.1]{ddr16} (and its proof in \cite{cegm17}) that Thm.~\ref{thm:gen-conj} is true if and only if $n_\tau<p-1$ for all $\tau$, i.e. precisely if and only if $\ovr{r}$ is weakly generic. This shows that Hypo.~\ref{hypo:generic} is optimal for Thm.~\ref{thm:gen-conj} to hold when $e=1$. On the other hand, the calculations in \cite[\S8.1--8.4]{ddr16} show that Thm.~\ref{thm:mainthm} should still hold when $e=1$, $f=2$ and $\ovr{r}$ is not weakly generic. Although it follows from these calculations that one of the packets of weights $P_w$ with $\sum_\tau w_\tau = 1$ will be empty and the other one will have cardinality 2. This contradicts Prop.~\ref{prop:P-w-size-weak-case} and shows a crucial distinction between the weakly and non-generic cases for $f=2$ and $e=1$.

In \cite[\S7.5]{ste20} some conjectural formulae for $J_\sigma^\mathrm{AH}(\chi_1,\chi_2)$ are presented for $e=2$ and arbitrary $f$ -- see, for example, Conj.~7.5.12 in \cite{ste20}. These formulae suggest that when $e=2$ the correct condition on $\ovr{r}$ for Thm.~\ref{thm:gen-conj} to hold should be $n_\tau\in [2,p-1]$ for all $\tau$; whereas weak genericity requires $n_\tau \in [2,p-2]$ for all $\tau$ when $e=2$. Therefore, these calculations suggest that there are situations in which a version of Thm.~\ref{thm:gen-conj} may be true under weaker assumptions than the ones made in this paper.
 
\printbibliography

\end{document}